\documentclass[11pt]{article}
\usepackage[utf8]{inputenc}
\inputencoding{utf8}

\usepackage[margin=1in]{geometry}

\usepackage{makecell}
\usepackage{tabularx}
\usepackage{fancyvrb}
\usepackage{alphalph}
\usepackage{amssymb}
\usepackage{url}
\usepackage{enumitem}
\usepackage{amsthm}
\usepackage{pgf, tikz}
\usepackage{caption}
\usepackage{subcaption}

\setitemize{noitemsep,topsep=0pt,parsep=0pt,partopsep=0pt}
\setenumerate{noitemsep,topsep=0pt,parsep=0pt,partopsep=0pt}

\newtheorem{thm}{Theorem}

\newtheorem{con}[thm]{Conjecture}
\newtheorem{que}[thm]{Question}
\newtheorem{prop}[thm]{Proposition}

\newtheorem{obs}[thm]{Observation}
\newtheorem{fact}[thm]{Fact}

\newtheorem{cor}[thm]{Corollary}

\pgfdeclarelayer{edgelayer}
\pgfdeclarelayer{nodelayer}
\pgfsetlayers{edgelayer,nodelayer,main}

\tikzstyle{none}=[inner sep=0pt]
\definecolor{hexcolor0xf81e1c}{rgb}{0.973,0.118,0.110}
\definecolor{hexcolor0x3c00ff}{rgb}{0.235,0.000,1.000}

\tikzstyle{whitevertex}=[circle,fill=white,draw=black, scale = 0.7]
\tikzstyle{redvertex}=[circle,fill=hexcolor0xf81e1c,draw=black, scale = 0.7]
\tikzstyle{bluevertex}=[circle,fill=hexcolor0x3c00ff,draw=black, scale = 0.7]
\tikzstyle{greenvertex}=[circle,fill=green,draw=black, scale=0.7]
\tikzstyle{purplevertex}=[circle,fill=magenta,draw=black, scale=0.7]
\tikzstyle{grayvertex}=[circle,fill=white,draw=gray, scale=0.7]
\tikzstyle{blackvertex}=[circle,fill=black,draw=black, scale=0.7]

\tikzstyle{textbox}=[rectangle,fill=none,draw=none]
\tikzstyle{box}=[rectangle,fill=none,draw=black]

\tikzstyle{grayedge}=[draw=gray]
\tikzstyle{blueedge}=[draw=blue]
\tikzstyle{rededge}=[draw=red]
\tikzstyle{edge}=[draw=black]

\tikzstyle{vertex}=[circle, ,fill=white,draw=black, scale=0.66]

\tikzstyle{10circle}=[circle, scale=10.0,draw=black] 
\tikzstyle{10oval}=[ellipse, scale=10.0,draw=black]

\title{Eternal Domination and Clique Covering}

\author{Gary MacGillivray \thanks{Funded by a Discovery Grant from the Natural Sciences and Engineering Research Council of Canada.} \thanks{Department of Mathematics and Statistics, University of Victoria, Victoria BC, Canada. \newline \indent \hspace{4pt} E-mail addresses: \texttt{gmacgill@uvic.ca, kieka@uvic.ca, virgilev@uvic.ca}} \and C. M. Mynhardt $^{* \dagger}$ \and Virgélot Virgile $^\dagger$}

\date{\today}

\begin{document}

\maketitle

\begin{abstract}
We study the relationship between the eternal domination number of a graph and its clique covering number using both large-scale computation and analytic methods. In doing so, we answer two open questions of Klostermeyer and Mynhardt. We show that the smallest graph having its eternal domination number less than its clique covering number has $10$ vertices. We determine the complete set of $10$-vertex and $11$-vertex graphs having eternal domination numbers less than their clique covering numbers. We show that the smallest triangle-free graph with this property has order $13$, as does the smallest circulant graph. We describe a method to generate an infinite family of triangle-free graphs and an infinite family of circulant graphs with eternal domination numbers less than their clique covering numbers. We also consider planar graphs and cubic graphs. Finally, we show that for any integer $k \geq 2$ there exist infinitely many graphs having domination number and eternal domination number equal to $k$ containing dominating sets which are not eternal dominating sets. \newline



\end{abstract}

%===== ===== ===== ===== ===== ===== ===== ===== ===== ===== ===== ===== ===== ===== ===== ===== ===== ===== ===== =====

\section{Introduction}

%===== ===== ===== ===== ===== ===== ===== ===== ===== ===== ===== ===== ===== ===== ===== ===== ===== ===== ===== =====

The eternal domination game, played on graphs, was introduced by Burger, Cockayne, Grundlingh, Mynhardt, Van Vuuren and Winterbach \cite{burger2004infinite}. The game is played between two players that alternate turns: a defender who controls a set of guards, and an attacker. To start the game, the defender chooses a dominating set $D_0 \subseteq V$ such that $|D_0|=k$ on which to place the guards (at most one guard on each vertex). At each time $t=1, 2, 3, \ldots$ the attacker selects a vertex $v$ on which there is no guard; we say the attacker \textit{attacks} $v$. The defender responds by moving a guard on a neighbour of $v$ to $v$; we say the defender \textit{defends} $v$. The guards (or the defender) win if they are able to respond to the sequence of attacks, that is, if they can maintain a dominating set throughout the game; otherwise, the attacker wins. In other words, the attacker wins if at some time $t$ there is no guard in the neighbourhood of some vertex. The eternal domination number of a graph $G$, denoted by $\gamma^\infty(G)$, is the minimum number of guards necessary to respond to any sequence of attacks on $G$. For a survey on the eternal domination game and its variants, see \cite{klostermeyer2016protecting} or \cite{klostermeyer2020eternal}.

For graphs belonging to certain families, the eternal domination number is closely related to another well-studied parameter: the clique covering number. The \emph{clique covering number} of a graph $G$, denoted by $\theta(G)$, is the minimum cardinality of a clique covering of $G$, that is, a partition $\{V_1, V_2, \ldots, V_k\}$ of $V(G)$ such that each $V_i$ induces a clique. Observe that, as stated in \cite{burger2004infinite}, if we consider a minimum clique covering of $G$, place a guard in each clique of the covering and play the game independently in each clique, each guard is able to respond to any sequence of attacks on its respective clique. This strategy shows that the clique covering number of a graph is an upper bound on its eternal domination number. In fact, $\gamma^\infty(G)=\theta(G)$ if $G$ is a perfect graph \cite{burger2004infinite}, an outerplanar graph \cite{anderson2007maximum}, or a graph with $\theta(G) \leq 3$ \cite{burger2004infinite}; a complete characterization of the graphs $G$ with $\gamma^\infty(G)=\theta(G)$ is yet to be found. Goddard, Hedetniemi and Hedetniemi \cite{goddard2005eternal} showed that there exist graphs with $\gamma^\infty<\theta$; they gave an $11$-vertex graph as example: the complement of the Gr\"otzsch graph. Klostermeyer and Mynhardt \cite{klostermeyer2016protecting} asked whether that $11$-vertex graph is the smallest graph, in terms of the number of vertices, with this property. They further asked whether there exist graphs with $\gamma^\infty < \theta$ in some particular graph classes, for example planar graphs.

This paper is organised as follows. In Sections $2$ and $3$, we review the necessary background. In Section $4$, we describe a large-scale computation that shows, in Section $5.1$, that the smallest graph with eternal domination number less than its clique covering number has $10$ vertices and further determines the complete set of $10$-vertex and $11$-vertex graphs having eternal domination numbers less than their clique covering numbers. The computational results are supported by an analytic proof. In Section \ref{Section:GraphClasses}, we restrict our attention to triangle-free graphs, circulant graphs, planar graphs and cubic graphs. Using computation, we found that the smallest triangle-free graph with eternal domination number less than its clique covering number has $13$ vertices and that the smallest circulant graph with eternal domination number less than its clique covering number has $13$ vertices. We also consider a question (Question \ref{Conjecture:AlphaTheta}) of Klostermeyer and Mynhardt regarding triangle-free graphs. We verify that the smallest triangle-free graph with the described properties in Question \ref{Conjecture:AlphaTheta}, if it exists, has at least $15$ vertices. Our computations also show that all planar graphs on fewer than $12$ vertices, all $3$-connected planar graphs on fewer than $14$ vertices and all cubic graphs on fewer than $18$ vertices have eternal domination numbers equal to their clique covering numbers. Finally, in Section \ref{Section:DominationEternalDomination}, we consider another question (Question \ref{Question:GammaSetsGeneral}) and a conjecture (Conjecture \ref{Conjecture:GammaTheta}) of Klostermeyer and Mynhardt. We show that the answer to Question \ref{Question:GammaSetsGeneral} is no and verify Conjecture \ref{Conjecture:GammaTheta} for all graphs up to order $11$.

%===== ===== ===== ===== ===== ===== ===== ===== ===== ===== ===== ===== ===== ===== ===== ===== ===== ===== ===== =====

\section{Definitions}
The \emph{domination number} of a graph $G$, denoted by $\gamma(G)$, is the minimum cardinality of a dominating set of $G$, that is, a set $D \subseteq V(G)$ such that any vertex $v \in V(G) \backslash D$ has a neighbour in $D$. This means that, in the eternal domination game, the guards must always be located on the vertices of a dominating set of $G$ in order to defend a sequence of attack on $G$. A dominating set from which the guards can defend any sequence of attack is known as an \textit{eternal dominating set}.

The \emph{independence number} of $G$, denoted by $\alpha(G)$, is the maximum cardinality of an independent set of $G$, that is, a set $S \subseteq V(G)$ such that for any $u, v \in S$, $uv \not\in E(G)$. The \emph{clique number} of $G$, denoted by $\omega(G)$, is the maximum cardinality of a clique of $G$, where a clique is the complement of an independent set. This implies that $\alpha(G)=\omega(\overline{G})$ for any graph $G$ and its complement $\overline{G}$.

The \emph{chromatic number of $G$}, denoted by $\chi(G)$, is the minimum number of colours required to colour the vertices of $G$ so that no two adjacent vertices have the same colour. Observe that such a colouring with $k$ colours is obtained by partitionning $V(G)$ into $\{V_1, V_2, \ldots, V_k\}$ where each $V_i$ induces an independent set. For this reason, the clique covering number of a graph is equal to the chromatic number of its complement.	When the graph $G$ is clear from context, we use $n, \alpha, \gamma^\infty, \theta$ to denote respectively $|V(G)|$, $\alpha(G)$, $\gamma^\infty(G)$, $\theta(G)$.

A graph $G$ is said to be \emph{vertex critical} with respect to $\theta$ if $\theta(G-\{v\})=\theta(G)-1$ for any vertex $v \in V(G)$. A graph $G$ is said to be \emph{edge critical} with respect to $\theta$ if $\theta(G+\{uv\})=\theta(G)-1$ for any edge $uv \not\in E(G)$, where $G+\{uv\}$ is the graph obtained from $G$ by adding the missing edge $uv$. Since we only consider criticality with respect to $\theta$, we simply refer to $G$ being vertex or edge critical without explicitly referring to $\theta$. We say that a graph $G$ is \textit{critical} if $G$ is vertex critical and edge critical.

The \emph{circulant graph} $C_n[k_1, k_2, \ldots, k_l]$, where $\{k_1, k_2, \ldots, k_l\} \subseteq \mathbb{Z}^+$ and $1 \leq k_1 < k_2 < \ldots < k_l \leq \lfloor \frac{n}{2} \rfloor$, is the graph with vertex set $\{v_0, v_1, v_2, \ldots, v_{n-1}\}$ such that two vertices $v_i$ and $v_j$ are adjacent if and only if $i-j \equiv \pm k_p \pmod n$ for some $p \in \{1, 2, \ldots, l\}$. A \emph{perfect graph} is a graph such that the chromatic number of any of its induced subgraphs is equal to the clique number of that subgraph.

The \emph{bow tie product} of a graph $G$ with a graph $H$, denoted by $G \bowtie H$, is the graph with vertex set $\{(v_i, v_j): v_i \in V(G), v_j \in V(H)\}$, where two vertices $(v_i, v_j)$ and $(v_i', v_j')$ are adjacent if and only if one of the following conditions holds: $v_i v_i' \in E(G)$ and $v_j = v_j'$, or $v_i v_i' \in E(G)$ and $v_j v_j' \in E(H)$ (see Figure \ref{Figure:BowtieProduct}). As shown in Figure \ref{Figure:BowtieProduct}, $G \bowtie H \ncong H \bowtie G$ in general.
	
	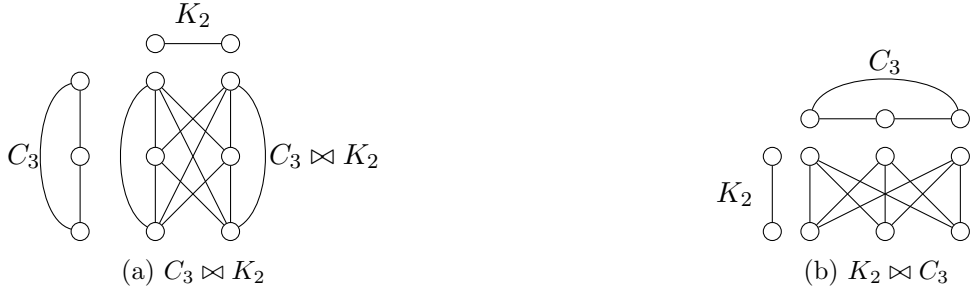
\begin{figure}[h!]
		\centering
		\begin{subfigure}[b]{0.45\textwidth}
			\centering
			\begin{tikzpicture}
				\begin{pgfonlayer}{nodelayer}
					\node [style=vertex] (0) at (-0.5, 1.5) {};
					\node [style=vertex] (1) at (0.5, 1.5) {};
					\node [style=vertex] (2) at (-1.5, 1) {};
					\node [style=vertex] (3) at (-1.5, 0) {};
					\node [style=vertex] (4) at (-1.5, -1) {};
					\node [style=vertex] (6) at (-0.5, 1) {};
					\node [style=vertex] (7) at (-0.5, 0) {};
					\node [style=vertex] (8) at (-0.5, -1) {};
					\node [style=vertex] (10) at (0.5, 1) {};
					\node [style=vertex] (11) at (0.5, 0) {};
					\node [style=vertex] (12) at (0.5, -1) {};
					\node [style=none] (16) at (-2.25, 0) {$C_3$};
					\node [style=none] (17) at (0, 1.875) {$K_2$};
					\node [style=none] (18) at (1.75, 0) {$C_3 \bowtie K_2$};
					\node [style=none] (19) at (2.25, 0) {};
				\end{pgfonlayer}
				\begin{pgfonlayer}{edgelayer}
					\draw (2) to (3);
					\draw (3) to (4);
					\draw (0) to (1);
					\draw (6) to (7);
					\draw (7) to (8);
					\draw (10) to (11);
					\draw (11) to (12);
					\draw (6) to (11);
					\draw (7) to (10);
					\draw (7) to (12);
					\draw (8) to (11);
					\draw [bend right=75, looseness=0.75] (2) to (4);
					\draw [bend right=60, looseness=0.75] (6) to (8);
					\draw [bend left=60, looseness=0.75] (10) to (12);
					\draw (6) to (12);
					\draw (10) to (8);
				\end{pgfonlayer}
			\end{tikzpicture}
			
			\caption{$C_3 \bowtie K_2$}
			\label{Figure:C3K2}	
		\end{subfigure}
		\hfill
		\begin{subfigure}[b]{0.45\textwidth}
			\centering
			\begin{tikzpicture}
			\begin{pgfonlayer}{nodelayer}
			\node [style=vertex] (19) at (-1.5, 1) {};
			\node [style=vertex] (20) at (-1.5, 0) {};
			\node [style=vertex] (21) at (1, 1.5) {};
			\node [style=vertex] (22) at (0, 1.5) {};
			\node [style=vertex] (23) at (-1, 1.5) {};
			\node [style=none] (30) at (0, 2.25) {$C_3$};
			\node [style=none] (31) at (-2, 0.5) {$K_2$};
			\node [style=vertex] (32) at (-1, 1) {};
			\node [style=vertex] (33) at (0, 1) {};
			\node [style=vertex] (34) at (1, 1) {};
			\node [style=vertex] (35) at (-1, 0) {};
			\node [style=vertex] (36) at (0, 0) {};
			\node [style=vertex] (37) at (1, 0) {};
			\node [style=none] (38) at (2, 0.5) {};
			\end{pgfonlayer}
			\begin{pgfonlayer}{edgelayer}
			\draw (21) to (22);
			\draw (22) to (23);
			\draw (19) to (20);
			\draw [bend right=75, looseness=0.75] (21) to (23);
			\draw (32) to (35);
			\draw (33) to (36);
			\draw (34) to (37);
			\draw (32) to (36);
			\draw (33) to (35);
			\draw (33) to (37);
			\draw (34) to (36);
			\draw (32) to (37);
			\draw (34) to (35);
			\end{pgfonlayer}
			\end{tikzpicture}
				
			\caption{$K_2 \bowtie C_3$}
			\label{Figure:K2C3}	
		\end{subfigure}
		\caption{Bow tie product of two graphs.}
		\label{Figure:BowtieProduct}
	\end{figure}

We refer to a graph $G$ as being a \textit{smallest graph having some property $\mathcal{P}$} if no graph $H$ on fewer than $|V(G)|$ vertices has property $\mathcal{P}$.

%===== ===== ===== ===== ===== ===== ===== ===== ===== ===== ===== ===== ===== ===== ===== ===== ===== ===== ===== =====
	
\section{Preliminary results} \label{Section:PreliminaryResults}
	
It is straightforward to see that the eternal domination number of a graph is the sum of the eternal domination number of its components; therefore, we restrict our attention to connected graphs. 
	
We begin with the following observation that we may obtain upper and lower bounds on $\gamma^\infty(G)$ by considering the problem on the induced and spanning subgraphs of $G$.
	
	\begin{fact} \label{Fact:Spanning}
		Let $G$ be a graph and let $H$ be a subgraph of $G$.
		\begin{enumerate}[label=(\alph*)]
			\item \rm{\cite{klostermeyer2009eternal}} \textit{If $H$ is an induced subgraph of $G$, then $\gamma^\infty(G) \geq \gamma^\infty(H)$.}
			\item \rm{\cite{anderson2007maximum}} \textit{If $H$ is a spanning subgraph of $G$, then $\gamma^\infty(G) \leq \gamma^\infty(H)$.}
		\end{enumerate}
	\end{fact}

	 The reader can easily check that $\gamma^\infty(K_n)=1$ and $\gamma^\infty(\overline{K_n})=n$. Hence, by considering a minimum clique covering of $G$ and playing the game independently on each clique we obtain an upper bound on $\gamma^\infty(G)$. Likewise, if we consider a maximum independent set $S$ of $G$ and play the game on the subgraph induced by $S$, we obtain a lower bound on $\gamma^\infty(G)$. In consequence, we have the following fact of Burger \textit{et al.}
	
	\begin{fact} [\cite{burger2004infinite}] \label{Fact:Bounds}
		For any graph $G$, $\alpha(G) \leq \gamma^\infty(G) \leq \theta(G)$.
	\end{fact}
	
	The Weak Perfect Graph Theorem \cite{lovasz1972characterization, lovasz1972normal} states that the complement of a perfect graph is perfect. As a direct consequence we have the following corollary.
	
	\begin{cor} [\cite{burger2004infinite}] \label{Corollary:Perfect}
		For any perfect graph $G$, $\alpha(G)=\gamma^\infty(G)=\theta(G)$.
	\end{cor}
	
	The inequalities in Fact \ref{Fact:Bounds} can both be strict. Observe that the $5$-cycle is a graph with $\alpha=2<\gamma^\infty=3$ and is the smallest such graph. As for the second inequality, the first proof of the existence of a graph with $\gamma^\infty<\theta$ is due to Goddard, Hedetniemi and Hedetniemi \cite{goddard2005eternal} and follows from Theorem \ref{Theorem:GoddardSmallest} and its generalisation, Theorem \ref{Theorem:Alpha+1}, which shows that the eternal domination number of a graph is bounded above by a function of the independence number of the graph.

	\begin{thm} [\cite{goddard2005eternal}] \label{Theorem:GoddardSmallest}
		If $G$ is a graph such that $\alpha(G)=2$, then $\gamma^\infty(G) \leq 3$.
	\end{thm}

	\begin{thm} [\cite{klostermeyer2007eternal}] \label{Theorem:Alpha+1}
		For any graph $G$, $\gamma^\infty(G) \leq {\alpha(G)+1 \choose 2}$.
	\end{thm}
	
	It is known that for any integer $k \geq 2$ there exists a triangle-free graph $G$ with chromatic number $k$. The first known construction of a family containing such graphs with arbitrarily large chromatic numbers is due to Blanche Descartes \cite{descartes1947three, descartes1954solution}. Mycielski \cite{mycielski1955coloriage} described the construction of a family $\mathcal{F}=\{M_2, M_3, M_4, \ldots\}$ of triangle-free graphs starting with $M_2 = K_2$, where, for each $k \geq 3$, the graph $M_k$ is obtained from the graph $M_{k-1}$ and has chromatic number $k$ (see $M_4$ in Figure \ref{Figure:Grotzsch}).
	
	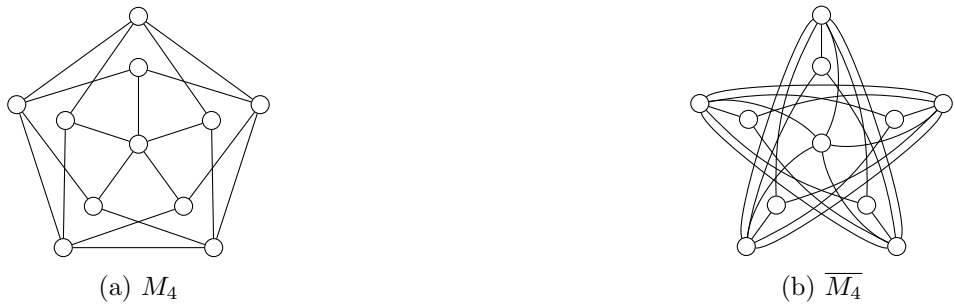
\begin{figure}[h!]
		\centering
		\begin{subfigure}[b]{0.45\textwidth}
			\centering
			\begin{tikzpicture}
			\begin{pgfonlayer}{nodelayer}
			\node [style=vertex] (0) at (0, 3.0776) {};
			\node [style=vertex] (1) at (1.618, 1.9021) {};
			\node [style=vertex] (2) at (1, 0) {};
			\node [style=vertex] (3) at (-1, 0) {};
			\node [style=vertex] (4) at (-1.618, 1.9021) {};
			\node [style=vertex] (5) at (0, 2.3971) {};
			\node [style=vertex] (6) at (0.9708, 1.6918) {};
			\node [style=vertex] (7) at (0.6, 0.5505) {};
			\node [style=vertex] (8) at (-0.6, 0.5505) {};
			\node [style=vertex] (9) at (-0.9708, 1.6918) {};
			\node [style=vertex] (10) at (0, 1.3763) {};
			\end{pgfonlayer}
			\begin{pgfonlayer}{edgelayer}
			\draw (0) to (1);
			\draw (1) to (2);
			\draw (2) to (3);
			\draw (3) to (4);
			\draw (4) to (0);
			\draw (5) to (10);
			\draw (6) to (10);
			\draw (7) to (10);
			\draw (8) to (10);
			\draw (9) to (10);
			\draw (0) to (6);
			\draw (0) to (9);
			\draw (1) to (5);
			\draw (1) to (7);
			\draw (2) to (6);
			\draw (2) to (8);
			\draw (3) to (7);
			\draw (3) to (9);
			\draw (4) to (8);
			\draw (4) to (5);
			\end{pgfonlayer}
			\end{tikzpicture}
			
			\caption{$M_4$}
			\label{Figure:Grotzsch}	
		\end{subfigure}
		\hfill
		\begin{subfigure}[b]{0.45\textwidth}
			\centering
			\begin{tikzpicture}
			\begin{pgfonlayer}{nodelayer}
			\node [style=vertex] (0) at (0, 3.0776) {};
			\node [style=vertex] (1) at (1.618, 1.9021) {};
			\node [style=vertex] (2) at (1, 0) {};
			\node [style=vertex] (3) at (-1, 0) {};
			\node [style=vertex] (4) at (-1.618, 1.9021) {};
			\node [style=vertex] (5) at (0, 2.3971) {};
			\node [style=vertex] (6) at (0.9708, 1.6918) {};
			\node [style=vertex] (7) at (0.6, 0.5505) {};
			\node [style=vertex] (8) at (-0.6, 0.5505) {};
			\node [style=vertex] (9) at (-0.9708, 1.6918) {};
			\node [style=vertex] (10) at (0, 1.3763) {};
			\end{pgfonlayer}
			\begin{pgfonlayer}{edgelayer}
			\draw [bend right=15] (0) to (8);
			\draw [bend right=45, looseness=0.25] (0) to (3);
			\draw [bend left=15] (0) to (7);
			\draw [bend left=45, looseness=0.25] (0) to (2);
			\draw [bend left=345] (1) to (9);
			\draw [bend left=315, looseness=0.25] (1) to (4);
			\draw [bend left=15] (1) to (8);
			\draw [bend left=45, looseness=0.25] (1) to (3);
			\draw [bend left=15] (2) to (9);
			\draw [bend right=315, looseness=0.25] (2) to (4);
			\draw [bend right=15] (2) to (5);
			\draw [bend right=15] (3) to (6);
			\draw [bend left=15, looseness=0.75] (3) to (5);
			\draw [bend left=15] (4) to (6);
			\draw [bend right=15] (4) to (7);
			\draw [bend right, looseness=0.75] (10) to (0);
			\draw [bend right, looseness=0.75] (10) to (4);
			\draw [bend right, looseness=0.75] (10) to (3);
			\draw [bend right, looseness=0.75] (10) to (2);
			\draw [bend right, looseness=0.75] (10) to (1);
			\draw (0) to (5);
			\draw (4) to (9);
			\draw (3) to (8);
			\draw (2) to (7);
			\draw (1) to (6);
			\end{pgfonlayer}
			\end{tikzpicture}
			
			\caption{$\overline{M_4}$}
			\label{Figure:GrotzschComplement}	
		\end{subfigure}
		\caption{The Grötzsch graph (a) and its complement (b).}
		\label{Figure:GrotzschGraph}
	\end{figure}

	\begin{cor} [\cite{goddard2005eternal}]
		For any integer $k \geq 4$, $\gamma^\infty(\overline{M_k})<\theta(\overline{M_k})$.
	\end{cor}

	Our primary purpose in this paper is to study the relationship between the eternal domination number and the clique covering number for some specific graphs; however, Theorem \ref{Theorem:Alpha+1} implies a nice result on the problem regarding general random graphs that is worth mentioning.
	
	\begin{thm} [\cite{alon1992probabilistic}, Theorem $3.1$] \label{Theorem:RandomGraphs}
		For almost all graphs $G$, $\alpha(G) \leq (2+o(1)) \log_2 n$ and $\theta(G) \geq (1+o(1)) \frac{n}{2 \log_2 n}$.
	\end{thm}

	\begin{cor}
		For almost all graphs $G$, $\gamma^\infty(G) < \theta(G)$.
	\end{cor}

	\begin{proof}
		We know from Theorem \ref{Theorem:RandomGraphs} that $\alpha(G) \leq (2+o(1)) \log_2 n$ for almost all graphs $G$. Together with Theorem \ref{Theorem:Alpha+1}, this implies that $\gamma^\infty(G) \leq (2+o(1)) (\log_2 n)^2 $ for almost all graphs $G$. Since $\lim_{n \to \infty} \frac{2 (\log_2 n)^2}{n/(2 \log_2 n)} =0$, we conclude that $\gamma^\infty(G)<\theta(G)$ for almost all graphs $G$.
	\end{proof}
	
	\begin{thm} [\cite{chvatal1974minimality}]
		The Grötzsch graph (Figure \ref{Figure:Grotzsch}) is the unique smallest triangle-free graph with chromatic number at least $4$.
	\end{thm}

	\begin{cor}
		The complement of the Grötzsch graph is the unique smallest graph with independence number $2$ and clique covering number at least $4$.
	\end{cor}

	Klostermeyer and Mynhardt \cite{klostermeyer2016protecting, klostermeyer2020eternal} posed the following questions.

	\begin{que} [\cite{klostermeyer2016protecting}] \label{Question:Smallest}
		Is the complement of the Grötzsch graph, a graph of order $11$, the smallest graph with eternal domination number less than its clique covering number?
	\end{que}

	\begin{que} [\cite{klostermeyer2020eternal}] \label{Question:GammaSets}
		Let $G$ be a graph with $\gamma(G)=\gamma^\infty(G)$. Is any minimum dominating set of $G$ an eternal dominating set of $G$?
	\end{que}

	We answer Question \ref{Question:Smallest} in Section \ref{Section:GeneralGraphs} by finding two graphs of order $10$ with this property. We also show that the answer to Question \ref{Question:GammaSets} is no by constructing an infinite family of counterexamples.

%===== ===== ===== ===== ===== ===== ===== ===== ===== ===== ===== ===== ===== ===== ===== ===== ===== ===== ===== =====

	\section{Computational methods} \label{Section:Algorithms}
	Consider the following decision problem.
	\begin{center}
		\fbox{\parbox{6in}{
				ETERNAL DOMINATION \\
				INSTANCE: A graph $G$ and an integer $k>0$. \\
				QUESTION: Is $\gamma^\infty(G) \leq k?$}}
	\end{center}
	
	Although the precise complexity of ETERNAL DOMINATION is not known for general graphs, Klostermeyer and MacGillivray \cite{klostermeyer2016dynamic} showed that the problem can be solved in time exponential in $n$. Their algorithm is described for a variant of the eternal domination game. Algorithm $1$ is adapted to the version of the game considered in this paper. Klostermeyer also studied the complexity of the problem under narrow assumptions on the sequence of attacks \cite{klostermeyer2007complexity}.
	
	The precise complexity of ETERNAL DOMINATION is known for graphs belonging to certain graph classes, for example perfect graphs. We know from Corollary \ref{Corollary:Perfect} that the eternal domination number of a perfect graph is equal to its independence number and hence its clique covering number. It is well known \cite{grtschel1988geometric} that the independence number of a perfect graph can be found in polynomial time in $n$. As a result, ETERNAL DOMINATION is in P for perfect graphs.
	
	\begin{table}[h]
		\begin{center}
			\begin{tabular}{p{15cm}}
				\hline
				\textbf{Algorithm 1} (\cite{klostermeyer2016dynamic})\textbf{.} Determine whether $\gamma^\infty(G) \leq k$. \\
				\hline
				Given $G$, construct an arc-coloured digraph $D$ as follows.
				\begin{enumerate}
					\item The vertex set of $D$ is the set of $k$-vertex dominating sets of $G$. The set of colours is $V(G)$. There is an arc from $X$ to $Y$ of colour $v$ when $Y-X=\{v\}$, and $v$ is adjacent in $G$ to the unique vertex $w \in X-Y$.
					\item Delete any vertex $X$ which is not the origin of an arc coloured $x$ for some $x \in V(G) \backslash X$.
					\item Repeat Step $2$ until no further vertices can be deleted.
				\end{enumerate}
				If $D$ is the null graph, then $\gamma^\infty(G) > k$; otherwise, $\gamma^\infty(G) \leq k$. \\
				\hline
			\end{tabular}
		\end{center}
	\end{table}

	\vspace{-12pt}
	In the following section, we report on a large-scale computation performed on a PowerEdge R$7425$ server equipped with two AMD EPYC processors and totalling $64$ threaded CPU cores for the purposes of verifying some of the propositions in the paper. The computations were done in Python using NetworkX and Algorithm $1$ along with a few basic graph algorithms. 

	In our search, we often needed to generate the set of all graphs belonging to some particular class; to this end, we used NAUTY \cite{mckay2014practical} (version $2.7001$) and PLANTRI \cite{brinkmann2007fast, brinkmann2007plantri} (version $5.2$). On one hand, using NAUTY, we were able to generate the set of all graphs, the set of all triangle-free graphs, and the set of all cubic graphs of order $n$ for some given values of $n$. On the other hand, using PLANTRI, we were able to generate the set of all planar graphs of order $n$ for some given values of $n$.
	
	Since we studied the relationship between the eternal domination number and other graph parameters such as the domination number, the independence number and the clique covering number, we also found the value of each of these parameters for each graph in our computation. 
	
	We computed the independence number of each graph by computing the clique number of its complement using the built-in functions \verb|complement()| and \verb|graph_clique_number()| from the Python package NetworkX. As for the clique covering number, we computed this value using three different approaches. For general graphs, we started by enumerating the set of maximal cliques of the graph using the built-in function \verb|find_cliques()| from NetworkX. Using a naive method and an integer programming method, we found the size of the smallest subset of the set of maximal cliques which covers all the vertices of the graph. To solve the integer programs, we used the Python packages \verb|PuLP| and \verb|MIP| and they both agreed on the results. Our naive method is on average faster than the integer programming method on the set of graphs of up to order $11$. When it is known in advance that the graph is triangle-free, we first found the size of a maximum matching in the graph using the function \verb|max_weight_matching()| from NetworkX and then substracted this number from the order of the graph. As for the domination number of the graph, we found this value by first generating the $k$-vertex subsets of the $n$-vertex set. Then, using the built-in function \verb|is_dominating_set| from NetworkX, we checked whether the graph had a dominating set of that size.
	
	Now, to find the the eternal domination number of the graph, we first compared its independence number to its clique covering number using the algorithms described above. If these parameters are equal, then they are also equal to the eternal domination number. Otherwise, we computed the eternal domination number using Algorithm $1$. The largest single computation took approximately $216$ CPU days before the results (Table $2$) were found. The class of graphs on which our computations were the slowest is the set of maximal triangle-free graphs with independence numbers less than their clique covering numbers.
	
%===== ===== ===== ===== ===== ===== ===== ===== ===== ===== ===== ===== ===== ===== ===== ===== ===== ===== ===== =====
	
	\section{Main results} \label{Section:MainResults}

	\subsection{General graphs} \label{Section:GeneralGraphs}
	
	In this section, we show that the graphs $G_1$ and $G_2$ depicted in Figure \ref{Figure:SmallestCounterExample} are the smallest graphs having their eternal domination numbers less than their clique covering numbers. Observe that $G_2$ is obtained from $G_1$ by adding the edge $(67)$.
	
	\begin{figure}[h!]
		\centering
		\begin{subfigure}[b]{0.45\textwidth}
			\centering
			\begin{tikzpicture}
				\begin{pgfonlayer}{nodelayer}
					\node [style=vertex] (9) at (0, 3.2859) {$0$};
					\node [style=vertex] (4) at (-1.3514, 2.6351) {$6$};
					\node [style=vertex] (6) at (-0.8639, 1.7489) {$4$};
					\node [style=vertex] (1) at (-1.6852, 1.1727) {$1$};
					\node [style=vertex] (7) at (-0.75, 0) {$2$};
					\node [style=vertex] (3) at (0.75, 0) {$7$};
					\node [style=vertex] (10) at (1.6852, 1.1727) {$8$};
					\node [style=vertex] (5) at (0.8639, 1.7489) {$5$};
					\node [style=vertex] (8) at (1.3514, 2.6351) {$3$};
					\node [style=vertex] (2) at (0, 2.2) {$9$};
				\end{pgfonlayer}
				\begin{pgfonlayer}{edgelayer}
					\draw (1) to (4);
					\draw (1) to (5);
					\draw (1) to (6);
					\draw (1) to (7);
					\draw (1) to (10);
					\draw (2) to (3);
					\draw (2) to (5);
					\draw (2) to (6);
					\draw (2) to (7);
					\draw (2) to (9);
					\draw (3) to (6);
					\draw (3) to (7);
					\draw (3) to (8);
					\draw (3) to (10);
					\draw (4) to (6);
					\draw (4) to (7);
					\draw (4) to (8);
					\draw (4) to (9);
					\draw (5) to (7);
					\draw (5) to (8);
					\draw (5) to (9);
					\draw (5) to (10);
					\draw (6) to (9);
					\draw (6) to (10);
					\draw (8) to (9);
					\draw (8) to (10);
				\end{pgfonlayer}
			\end{tikzpicture}
				
			\caption{$G_1$}
			\label{Figure:G1}	
		\end{subfigure}
		\hfill
		\begin{subfigure}[b]{0.45\textwidth}
			\centering
			\begin{tikzpicture}
			\begin{pgfonlayer}{nodelayer}
			\node [style=vertex] (9) at (0, 3.2859) {$0$};
			\node [style=vertex] (4) at (-1.3514, 2.6351) {$6$};
			\node [style=vertex] (6) at (-0.8639, 1.7489) {$4$};
			\node [style=vertex] (1) at (-1.6852, 1.1727) {$1$};
			\node [style=vertex] (7) at (-0.75, 0) {$2$};
			\node [style=vertex] (3) at (0.75, 0) {$7$};
			\node [style=vertex] (10) at (1.6852, 1.1727) {$8$};
			\node [style=vertex] (5) at (0.8639, 1.7489) {$5$};
			\node [style=vertex] (8) at (1.3514, 2.6351) {$3$};
			\node [style=vertex] (2) at (0, 2.2) {$9$};
			\end{pgfonlayer}
			\begin{pgfonlayer}{edgelayer}
			\draw (1) to (4);
			\draw (1) to (5);
			\draw (1) to (6);
			\draw (1) to (7);
			\draw (1) to (10);
			\draw (2) to (3);
			\draw (2) to (5);
			\draw (2) to (6);
			\draw (2) to (7);
			\draw (2) to (9);
			\draw [bend right=15, looseness=0.50] (3) to (4);
			\draw (3) to (6);
			\draw (3) to (7);
			\draw (3) to (8);
			\draw (3) to (10);
			\draw (4) to (6);
			\draw (4) to (7);
			\draw (4) to (8);
			\draw (4) to (9);
			\draw (5) to (7);
			\draw (5) to (8);
			\draw (5) to (9);
			\draw (5) to (10);
			\draw (6) to (9);
			\draw (6) to (10);
			\draw (8) to (9);
			\draw (8) to (10);
			\end{pgfonlayer}
			\end{tikzpicture}
						
			\caption{$G_2$}
			\label{Figure:G2}	
		\end{subfigure}
		\caption{Smallest graphs with $\gamma^\infty<\theta$.}
		\label{Figure:SmallestCounterExample}
	\end{figure}
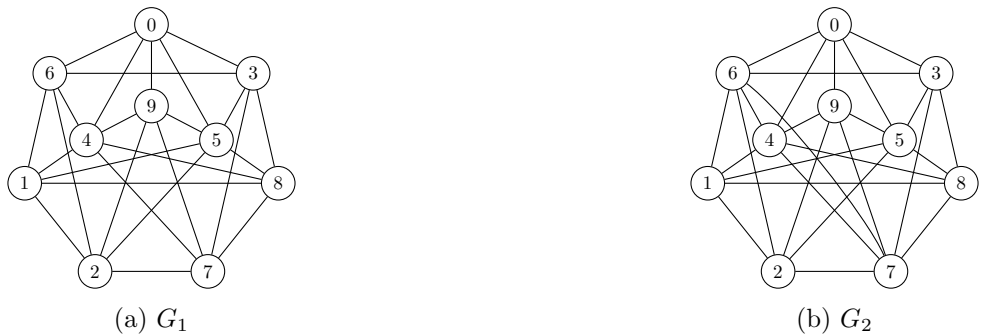

	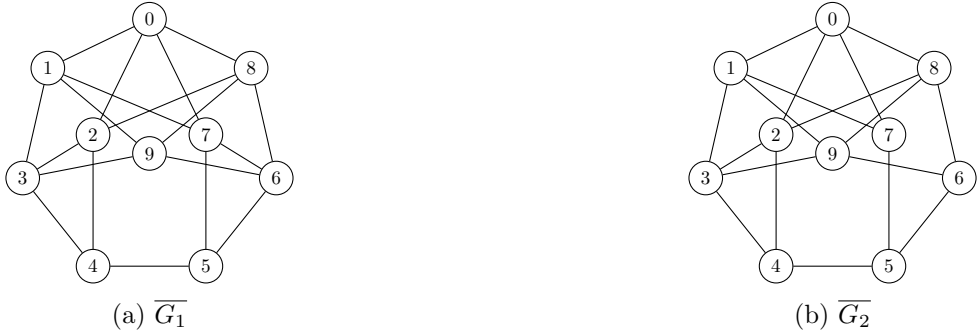
\begin{figure}[h!]
		\centering
		\begin{subfigure}[b]{0.45\textwidth}
			\centering
			\begin{tikzpicture}
				\begin{pgfonlayer}{nodelayer}
					\node [style=vertex] (10) at (0, 1.5) {$9$};
					\node [style=vertex] (1) at (0, 3.2859) {$0$};
					\node [style=vertex] (2) at (-1.3514, 2.6351) {$1$};
					\node [style=vertex] (3) at (-0.75, 1.75) {$2$};
					\node [style=vertex] (5) at (-0.75, 0) {$4$};
					\node [style=vertex] (9) at (1.3514, 2.6351) {$8$};
					\node [style=vertex] (8) at (0.75, 1.75) {$7$};
					\node [style=vertex] (6) at (0.75, 0) {$5$};
					\node [style=vertex] (4) at (-1.6852, 1.1727) {$3$};
					\node [style=vertex] (7) at (1.6852, 1.1727) {$6$};
				\end{pgfonlayer}
				\begin{pgfonlayer}{edgelayer}
					\draw (1) to (2);
					\draw (1) to (3);
					\draw (1) to (8);
					\draw (1) to (9);
					\draw (2) to (4);
					\draw (2) to (8);
					\draw (2) to (10);
					\draw (3) to (4);
					\draw (3) to (5);
					\draw (3) to (9);
					\draw (4) to (5);
					\draw (4) to (10);
					\draw (5) to (6);
					\draw (6) to (7);
					\draw (6) to (8);
					\draw (7) to (8);
					\draw (7) to (9);
					\draw (7) to (10);
					\draw (9) to (10);
				\end{pgfonlayer}
			\end{tikzpicture}
					
			\caption{$\overline{G_1}$}
			\label{Figure:G1Complement}	
		\end{subfigure}
		\hfill
		\begin{subfigure}[b]{0.45\textwidth}
			\centering
			\begin{tikzpicture}
				\begin{pgfonlayer}{nodelayer}
					\node [style=vertex] (10) at (0, 1.5) {$9$};
					\node [style=vertex] (1) at (0, 3.2859) {$0$};
					\node [style=vertex] (2) at (-1.3514, 2.6351) {$1$};
					\node [style=vertex] (3) at (-0.75, 1.75) {$2$};
					\node [style=vertex] (5) at (-0.75, 0) {$4$};
					\node [style=vertex] (9) at (1.3514, 2.6351) {$8$};
					\node [style=vertex] (8) at (0.75, 1.75) {$7$};
					\node [style=vertex] (6) at (0.75, 0) {$5$};
					\node [style=vertex] (4) at (-1.6852, 1.1727) {$3$};
					\node [style=vertex] (7) at (1.6852, 1.1727) {$6$};
				\end{pgfonlayer}
				\begin{pgfonlayer}{edgelayer}
					\draw (1) to (2);
					\draw (1) to (3);
					\draw (1) to (8);
					\draw (1) to (9);
					\draw (2) to (4);
					\draw (2) to (8);
					\draw (2) to (10);
					\draw (3) to (4);
					\draw (3) to (5);
					\draw (3) to (9);
					\draw (4) to (5);
					\draw (4) to (10);
					\draw (5) to (6);
					\draw (6) to (7);
					\draw (6) to (8);
					\draw (7) to (9);
					\draw (7) to (10);
					\draw (9) to (10);
				\end{pgfonlayer}
			\end{tikzpicture}
			
			\caption{$\overline{G_2}$}
			\label{Figure:G2Complement}	
		\end{subfigure}
		\caption{Complements of the graphs in Figure \ref{Figure:SmallestCounterExample}.}
		\label{Figure:SmallestComplement}
	\end{figure}
	
	We first show that these graphs have eternal domination numbers less than their clique covering numbers. To this end, we begin with the following facts; the proofs are easy and left to the reader. 
	
	\begin{fact} \label{Fact:IndependenceCliqueCovering}
		For the graphs $G_1, G_2, \overline{G_1}$ and $\overline{G_2}$ depicted in Figures \ref{Figure:SmallestCounterExample} and \ref{Figure:SmallestComplement},
		\begin{enumerate} [label=(\alph*)]
			\item $\alpha(G_1)=\alpha(G_2)=\omega(\overline{G_1})=\omega(\overline{G_2})=3$.
			\item $\theta(G_1)=\theta(G_2)=\chi(\overline{G_1})=\chi(\overline{G_2})=4$.
		\end{enumerate}
	\end{fact}
	
	\begin{fact} \label{Fact:Triangles}
		The graph $\overline{G_1}$ contains exactly six triangles, namely $\{2, 3, 4\}$, $\{5, 6, 7\}$, $\{0, 2, 8\}$, $\{0, 1, 7\}$, $\{1, 3, 9\}$ and $\{6, 8, 9\}$, any two of which share at most one vertex.
	\end{fact}
	
	\begin{fact} \label{Fact:Neighbourhood}
		The subgraph induced by the vertices in the open neighbourhood of each vertex of $\overline{G_1}$ is isomorphic to either $2 K_2$ or to $K_1 \cup K_2$.
	\end{fact}
	
	The following result can be verified by computer, but we include a proof as well.
	
	\begin{prop} \label{Proposition:SmallestGraph}
		For the graphs $G_1$ and $G_2$ depicted in Figure \ref{Figure:SmallestCounterExample}, $\gamma^\infty(G_1)=\gamma^\infty(G_2)<\theta(G_1)=\theta(G_2)$.
	\end{prop}
	
	\begin{proof}
		Since $G_1$ is a spanning subgraph of $G_2$, it suffices to show that $\gamma^\infty(G_1)=3$. Fact \ref{Fact:Spanning} will then imply that $\gamma^\infty(G_2)=3$. We do this by contradiction and assume that $\gamma^\infty(G_1) \geq 4$. We may also assume without loss of generality that three guards are initially located on the vertices of an independent set of $G_1$. This is because the attacker may sequentially attack all the vertices of a maximum independent set of $G_1$ and force a guard to be located on each of them. Let $k$ be the smallest integer such that there exists a sequence of attacks of length $k$ that the three guards cannot defend. Suppose the guards respond optimally to that sequence of attacks. At time  $t=k-1$, they will be located on the vertices $b, c$ and $d$, none of which is adjacent to vertex $a$ (see Figure \ref{Figure:GuardsConfiguration0} where a thick black edge corresponds to an edge in $G_1$ and a thin blue edge corresponds to an edge in $\overline{G_1}$). 
		
		\begin{figure}[h!]
			\centering
			\begin{subfigure}[b]{0.3\textwidth}
				\centering
				\begin{tikzpicture}[scale=0.5]
					\begin{pgfonlayer}{nodelayer}
						\node [style=vertex] (1) at (0, 2) {$a$};
						\node [style=vertex] (2) at (-2, 0) {$c$};
						\node [style=vertex] (3) at (0, 0.5) {$b$};
						\node [style=vertex] (4) at (2, 0) {$d$};
					\end{pgfonlayer}
					\begin{pgfonlayer}{edgelayer}
						\draw [style=blueedge] (1) to (2);
						\draw [style=blueedge] (1) to (3);
						\draw [style=blueedge] (1) to (4);
					\end{pgfonlayer}
				\end{tikzpicture}
			\end{subfigure}
			
			\caption{Configuration of the guards at time $t=k-1$.}
			\label{Figure:GuardsConfiguration0}
		\end{figure}
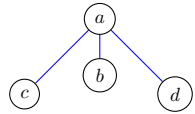
			
		Following Fact \ref{Fact:Neighbourhood}, we may assume that vertex $c$ is not adjacent to vertex $d$ and vertex $b$ is adjacent to both of $c$ and $d$. So, we colour the edge $cd$ blue and the edges $bc$ and $bd$ black (see Figure \ref{Figure:GuardsConfiguration01}). 

		\begin{figure}[h!]
			\centering
			\begin{subfigure}[b]{0.3\textwidth}
				\centering
				\begin{tikzpicture}[scale=0.5]
						\begin{pgfonlayer}{nodelayer}
							\node [style=vertex] (1) at (0, 2) {$a$};
							\node [style=vertex] (2) at (-2, 0) {$c$};
							\node [style=vertex] (3) at (0, 0.5) {$b$};
							\node [style=vertex] (4) at (2, 0) {$d$};
						\end{pgfonlayer}
						\begin{pgfonlayer}{edgelayer}
							\draw [style=blueedge] (1) to (2);
							\draw [style=blueedge] (1) to (3);
							\draw [style=blueedge] (1) to (4);
							\draw [style=ultra thick] (2) to (3);
							\draw [style=ultra thick] (3) to (4);
							\draw [style=blueedge, bend right=15] (2) to (4);
						\end{pgfonlayer}
					\end{tikzpicture}
			\end{subfigure}
			
			\caption{Configuration of the guards at time $t=k-1$.}
			\label{Figure:GuardsConfiguration01}
		\end{figure}
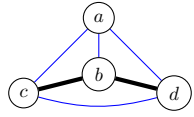		
		
		Note that in the previous time ($t=k-2$), the guards dominated all the vertices of $G$, in particular vertex $a$ was dominated. We now consider two cases depending on the location of the guards at time $t=k-2$.
		
		\begin{itemize}
			\item Case $1$: In the previous turn, a guard moved from a vertex $e$ to vertex $c$ (see Figure \ref{Figure:Loss2}). In this case, $e$ is adjacent to $a$ because $a$ was dominated at time $t=k-2$ but is undominated at time $t=k-1$, and $b$ has a neighbour $f$ which is non-adjacent to all of $c, d$ and $e$ (otherwise the guards would still be on a dominating set if the guard on $b$ moved to $c$ instead; see Figure \ref{Figure:Loss3}). This contradicts Fact \ref{Fact:Triangles} since the thin blue triangles $\{a, c, d\}$ and $\{f, c, d\}$ share two vertices ($c$ and $d$; see Figure \ref{Figure:Loss3}).
			
			\begin{figure}[h!]
				\centering
				\begin{subfigure}[b]{0.3\textwidth}
					\centering
					\begin{tikzpicture}[scale=0.5]
						\begin{pgfonlayer}{nodelayer}
							\node [style=vertex] (1) at (0, 2) {$a$};
							\node [style=vertex] (2) at (-2, 0) {$c$};
							\node [style=vertex] (3) at (0, 0.5) {$b$};
							\node [style=vertex] (4) at (2, 0) {$d$};
						\end{pgfonlayer}
						\begin{pgfonlayer}{edgelayer}
							\draw [style=blueedge] (1) to (2);
							\draw [style=blueedge] (1) to (3);
							\draw [style=blueedge] (1) to (4);
							\draw [style=ultra thick] (2) to (3);
							\draw [style=ultra thick] (3) to (4);
							\draw [style=blueedge, bend right=15] (2) to (4);
						\end{pgfonlayer}
					\end{tikzpicture}
					
					\caption{}
					\label{Figure:Loss1}
				\end{subfigure}
				\hfill
				\begin{subfigure}[b]{0.3\textwidth}
					\centering
					\begin{tikzpicture}[scale=0.5]
						\begin{pgfonlayer}{nodelayer}
							\node [style=vertex] (1) at (0, 2) {$a$};
							\node [style=vertex] (2) at (-2, 0) {$c$};
							\node [style=vertex] (3) at (0, 0.5) {$b$};
							\node [style=vertex] (4) at (2, 0) {$d$};
							\node [style=vertex] (5) at (-2, -2) {$e$};
						\end{pgfonlayer}
						\begin{pgfonlayer}{edgelayer}
							\draw [style=blueedge] (1) to (2);
							\draw [style=blueedge] (1) to (3);
							\draw [style=blueedge] (1) to (4);
							\draw [style=ultra thick] (2) to (3);
							\draw [style=ultra thick] (3) to (4);
							\draw [style=blueedge, bend right=15] (2) to (4);
							\draw [style=ultra thick] (2) to (5);
							\draw [style=ultra thick] (1) to (5);
						\end{pgfonlayer}
					\end{tikzpicture}
					
					\caption{}
					\label{Figure:Loss2}
				\end{subfigure}
				\hfill
				\begin{subfigure}[b]{0.3\textwidth}
					\centering
					\begin{tikzpicture}[scale=0.5]
						\begin{pgfonlayer}{nodelayer}
							\node [style=vertex] (1) at (0, 2) {$a$};
							\node [style=vertex] (2) at (-2, 0) {$c$};
							\node [style=vertex] (3) at (0, 0.5) {$b$};
							\node [style=vertex] (4) at (2, 0) {$d$};
							\node [style=vertex] (5) at (-2, -2) {$e$};
							\node [style=vertex] (6) at (0, -2) {$f$};
						\end{pgfonlayer}
						\begin{pgfonlayer}{edgelayer}
							\draw [style=blueedge] (1) to (2);
							\draw [style=blueedge] (1) to (3);
							\draw [style=blueedge] (1) to (4);
							\draw [style=ultra thick] (2) to (3);
							\draw [style=ultra thick] (3) to (4);
							\draw [style=blueedge, bend right=15] (2) to (4);
							\draw [style=ultra thick] (2) to (5);
							\draw [style=ultra thick] (1) to (5);
							\draw [style=ultra thick] (3) to (6);
							\draw [style=blueedge, bend right=15] (2) to (6);
							\draw [style=blueedge, bend left=15] (4) to (6);
							\draw [style=blueedge] (6) to (5);
						\end{pgfonlayer}
					\end{tikzpicture}
					
					\caption{}
					\label{Figure:Loss3}
				\end{subfigure}
				
				\caption{Configuration of the guards in Case $1$ of the proof of Proposition \ref{Proposition:SmallestGraph}.}
				\label{Figure:GuardsConfiguration}
			\end{figure}
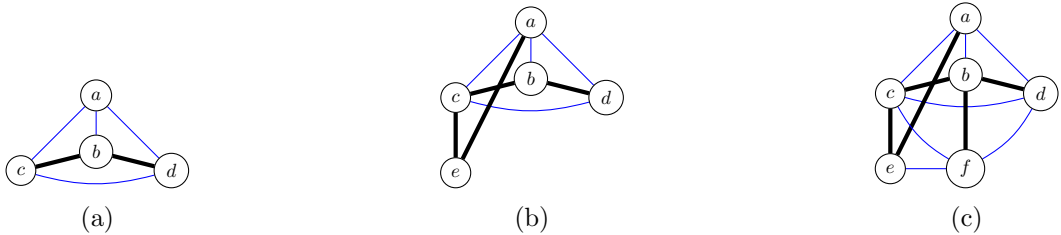
			
			\item Case $2$: In the previous turn, a guard moved from vertex $e$ to vertex $b$ (see Figure \ref{Figure:2Loss2}). In this case, $c$ has a neighbour $f$ which is non-adjacent to $b, d$ and $e$ (otherwise the guards would still be on a dominating set if the guard on $c$ moved to $b$ instead). For the same reason, $d$ has a neighbour $g$ which is not adjacent to $b, c$ and $e$, as shown in Figure \ref{Figure:2Loss3}. Following Fact \ref{Fact:Triangles}, $f$ must be adjacent to $g$, otherwise $\{f,b,g\}$ and $\{f,e,g\}$ would be two thin blue triangles sharing two vertices. By the same fact, $a$ must be adjacent to $f$ as $\{a,c,d\}$ and $\{a,f,d\}$ would share two vertices otherwise. For a similar reason, $a$ must be adjacent to $g$ as $\{a,c,g\}$ and $\{a,c,d\}$ would be two thin blue triangles sharing two vertices otherwise (see Figure \ref{Figure:2Loss4}).
			
			Now, the edges $ce$ and $de$ cannot be both thin blue, otherwise the thin blue triangles $\{a,c,d\}$ and $\{e,c,d\}$ would share two vertices. We assume without loss of generality that the edge $ce$ is black and we consider the two choices for the colour of the edge $de$. 
			
			\begin{figure}[h!]
				\centering
				\begin{subfigure}[b]{0.19\textwidth}
					\centering
					\begin{tikzpicture}[scale=0.45]
						\begin{pgfonlayer}{nodelayer}
							\node [style=vertex] (1) at (0, 2) {$a$};
							\node [style=vertex] (2) at (-2, 0) {$c$};
							\node [style=vertex] (3) at (0, 0.5) {$b$};
							\node [style=vertex] (4) at (2, 0) {$d$};
						\end{pgfonlayer}
						\begin{pgfonlayer}{edgelayer}
							\draw [style=blueedge] (1) to (2);
							\draw [style=blueedge] (1) to (3);
							\draw [style=blueedge] (1) to (4);
							\draw [style=ultra thick] (2) to (3);
							\draw [style=ultra thick] (3) to (4);
							\draw [style=blueedge, bend right=15] (2) to (4);
						\end{pgfonlayer}
					\end{tikzpicture}
					\caption{}
					\label{Figure:2Loss1}
				\end{subfigure}
				\hfill
				\begin{subfigure}[b]{0.19\textwidth}
					\centering
					\begin{tikzpicture}[scale=0.45]
						\begin{pgfonlayer}{nodelayer}
							\node [style=vertex] (1) at (0, 2) {$a$};
							\node [style=vertex] (2) at (-2, 0) {$c$};
							\node [style=vertex] (3) at (0, 0.5) {$b$};
							\node [style=vertex] (4) at (2, 0) {$d$};
							\node [style=vertex] (5) at (0, -2) {$e$};
						\end{pgfonlayer}
						\begin{pgfonlayer}{edgelayer}
							\draw [style=blueedge] (1) to (2);
							\draw [style=blueedge] (1) to (3);
							\draw [style=blueedge] (1) to (4);
							\draw [style=ultra thick] (2) to (3);
							\draw [style=ultra thick] (3) to (4);
							\draw [style=blueedge, bend right=15] (2) to (4);
							\draw [style=ultra thick] (3) to (5);
							\draw [style=ultra thick, bend left=30, looseness=1] (1) to (5);
						\end{pgfonlayer}
					\end{tikzpicture}
					\caption{}
					\label{Figure:2Loss2}
				\end{subfigure}
				\hfill
				\begin{subfigure}[b]{0.19\textwidth}
					\centering
					\begin{tikzpicture}[scale=0.45]
						\begin{pgfonlayer}{nodelayer}
							\node [style=vertex] (1) at (0, 2) {$a$};
							\node [style=vertex] (2) at (-2, 0) {$c$};
							\node [style=vertex] (3) at (0, 0.5) {$b$};
							\node [style=vertex] (4) at (2, 0) {$d$};
							\node [style=vertex] (5) at (0, -2) {$e$};
							\node [style=vertex] (6) at (-3, -3) {\tiny $f$};
							\node [style=vertex] (7) at (3, -3) {$g$};
						\end{pgfonlayer}
						\begin{pgfonlayer}{edgelayer}
							\draw [style=blueedge] (1) to (2);
							\draw [style=blueedge] (1) to (3);
							\draw [style=blueedge] (1) to (4);
							\draw [style=ultra thick] (2) to (3);
							\draw [style=ultra thick] (3) to (4);
							\draw [style=blueedge, bend right=15] (2) to (4);
							\draw [style=ultra thick] (3) to (5);
							\draw [style=ultra thick, bend left] (1) to (5);
							\draw [style=ultra thick] (2) to (6);
							\draw [style=blueedge] (3) to (6);
							\draw [style=blueedge] (5) to (6);
							\draw [style=blueedge, bend left=45] (4) to (6);
							\draw [style=ultra thick] (4) to (7);
							\draw [style=blueedge] (5) to (7);
							\draw [style=blueedge] (3) to (7);
							\draw [style=blueedge, bend right=45] (2) to (7);
						\end{pgfonlayer}
					\end{tikzpicture}
					
					\caption{}
					\label{Figure:2Loss3}
				\end{subfigure}
				\hfill
				\begin{subfigure}[b]{0.19\textwidth}
					\centering
					\begin{tikzpicture}[scale=0.45]
						\begin{pgfonlayer}{nodelayer}
							\node [style=vertex] (1) at (0, 2) {$a$};
							\node [style=vertex] (2) at (-2, 0) {$c$};
							\node [style=vertex] (3) at (0, 0.5) {$b$};
							\node [style=vertex] (4) at (2, 0) {$d$};
							\node [style=vertex] (5) at (0, -2) {$e$};
							\node [style=vertex] (6) at (-3, -3) {\tiny $f$};
							\node [style=vertex] (7) at (3, -3) {$g$};
						\end{pgfonlayer}
						\begin{pgfonlayer}{edgelayer}
							\draw [style=blueedge] (1) to (2);
							\draw [style=blueedge] (1) to (3);
							\draw [style=blueedge] (1) to (4);
							\draw [style=ultra thick] (2) to (3);
							\draw [style=ultra thick] (3) to (4);
							\draw [style=blueedge, bend right=15] (2) to (4);
							\draw [style=ultra thick] (3) to (5);
							\draw [style=ultra thick, bend left] (1) to (5);
							\draw [style=ultra thick] (2) to (6);
							\draw [style=blueedge] (3) to (6);
							\draw [style=blueedge] (5) to (6);
							\draw [style=blueedge, bend left=60, looseness=1.25] (4) to (6);
							\draw [style=ultra thick] (4) to (7);
							\draw [style=blueedge] (5) to (7);
							\draw [style=blueedge] (3) to (7);
							\draw [style=blueedge, bend right=60, looseness=1.25] (2) to (7);
							\draw [style=ultra thick, bend left=45] (1) to (7);
							\draw [style=ultra thick, bend right=45] (1) to (6);
							\draw [style=ultra thick, bend right=45, looseness=0.75] (6) to (7);
						\end{pgfonlayer}
					\end{tikzpicture}
					
					\caption{}
					\label{Figure:2Loss4}
				\end{subfigure}
				\hfill
				\begin{subfigure}[b]{0.19\textwidth}
					\centering
					\begin{tikzpicture}[scale=0.45]
						\begin{pgfonlayer}{nodelayer}
							\node [style=vertex] (1) at (0, 2) {$a$};
							\node [style=vertex] (2) at (-2, 0) {$c$};
							\node [style=vertex] (3) at (0, 0.5) {$b$};
							\node [style=vertex] (4) at (2, 0) {$d$};
							\node [style=vertex] (5) at (0, -2) {$e$};
							\node [style=vertex] (6) at (-3, -3) {\tiny $f$};
							\node [style=vertex] (7) at (3, -3) {$g$};
						\end{pgfonlayer}
						\begin{pgfonlayer}{edgelayer}
							\draw [style=blueedge] (1) to (2);
							\draw [style=blueedge] (1) to (3);
							\draw [style=blueedge] (1) to (4);
							\draw [style=ultra thick] (2) to (3);
							\draw [style=ultra thick] (3) to (4);
							\draw [style=blueedge, bend right=15] (2) to (4);
							\draw [style=ultra thick] (3) to (5);
							\draw [style=ultra thick, bend left] (1) to (5);
							\draw [style=ultra thick] (2) to (6);
							\draw [style=blueedge] (3) to (6);
							\draw [style=blueedge] (5) to (6);
							\draw [style=blueedge, bend left=60, looseness=1.25] (4) to (6);
							\draw [style=ultra thick] (4) to (7);
							\draw [style=blueedge] (5) to (7);
							\draw [style=blueedge] (3) to (7);
							\draw [style=blueedge, bend right=60, looseness=1.25] (2) to (7);
							\draw [style=ultra thick, bend left=45] (1) to (7);
							\draw [style=ultra thick, bend right=45] (1) to (6);
							\draw [style=ultra thick, bend right=45, looseness=0.75] (6) to (7);
							\draw [style=ultra thick] (2) to (5);
						\end{pgfonlayer}
					\end{tikzpicture}
					
					\caption{}
					\label{Figure:2Loss5}
				\end{subfigure}
				
				\caption{Configuration of the guards in Case $2$ of the proof of Proposition \ref{Proposition:SmallestGraph}.}
				\label{Figure:GuardsConfiguration2}
			\end{figure}
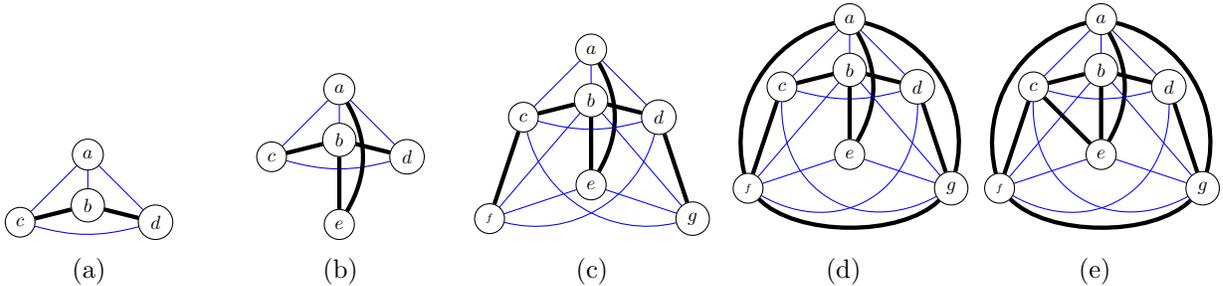			
			
			So, we obtain two blue/black colourings of the edges of the complete graph on $7$ vertices. One of the thin blue subgraphs obtained (Figure \ref{Figure:2Loss6}) must be an induced subgraph of $\overline{G_1}$. 
			
			\begin{figure}[h!]
				\centering
				\begin{tikzpicture}[scale=0.75]
					\begin{pgfonlayer}{nodelayer}
						\node [style=vertex] (1) at (-2.25, 0.5) {$a$};
						\node [style=vertex] (2) at (-3.75, 1.5) {$c$};
						\node [style=vertex] (3) at (-2.25, -0.5) {$b$};
						\node [style=vertex] (4) at (-0.75, 1.5) {$d$};
						\node [style=vertex] (5) at (-2.25, -1.5) {$e$};
						\node [style=vertex] (6) at (-0.75, -1.5) {$f$};
						\node [style=vertex] (7) at (-3.75, -1.5) {$g$};
						\node [style=vertex] (8) at (2.25, 0.5) {$a$};
						\node [style=vertex] (9) at (0.75, 1.5) {$c$};
						\node [style=vertex] (10) at (2.25, -0.5) {$b$};
						\node [style=vertex] (11) at (3.75, 1.5) {$d$};
						\node [style=vertex] (12) at (2.25, -1.5) {$e$};
						\node [style=vertex] (13) at (3.75, -1.5) {$f$};
						\node [style=vertex] (14) at (0.75, -1.5) {$g$};
					\end{pgfonlayer}
					\begin{pgfonlayer}{edgelayer}
						\draw [style=blueedge] (1) to (2);
						\draw [style=blueedge] (1) to (3);
						\draw [style=blueedge] (1) to (4);
						\draw [style=blueedge] (2) to (4);
						\draw [style=blueedge] (3) to (6);
						\draw [style=blueedge] (5) to (6);
						\draw [style=blueedge] (4) to (6);
						\draw [style=blueedge] (5) to (7);
						\draw [style=blueedge] (3) to (7);
						\draw [style=blueedge] (2) to (7);
						\draw [style=blueedge] (8) to (9);
						\draw [style=blueedge] (8) to (10);
						\draw [style=blueedge] (8) to (11);
						\draw [style=blueedge] (9) to (11);
						\draw [style=blueedge] (10) to (13);
						\draw [style=blueedge] (12) to (13);
						\draw [style=blueedge] (11) to (13);
						\draw [style=blueedge] (12) to (14);
						\draw [style=blueedge] (10) to (14);
						\draw [style=blueedge] (9) to (14);
						\draw [style=blueedge] (11) to (12);
					\end{pgfonlayer}
				\end{tikzpicture}
				
				\caption{Induced subgraphs of $\overline{G_1}$.}
				\label{Figure:2Loss6}
			\end{figure}
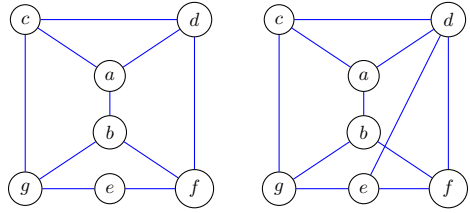
		
			The reader can check that $\overline{G_1}$ does not contain any of those graphs as an induced subgraph. Hence, we obtain a contradiction and conclude that $\gamma^\infty(G_1)=3$.
		\end{itemize}
	\vspace{-12pt}
	\end{proof}
	
%===== ===== ===== ===== ===== ===== ===== ===== ===== ===== ===== ===== ===== ===== ===== ===== ===== ===== ===== =====
	
	In the remaining part of this section, we verify that any graph with fewer than $10$ vertices has eternal domination number equal to its clique covering number. 
	
	\begin{prop} \label{Proposition:GammaInftyEqualTheta}
		For any graph $G$ of order $9$ or less, $\gamma^\infty(G)=\theta(G)$.
	\end{prop}
	
	Proposition \ref{Proposition:GammaInftyEqualTheta} was verified by computer; the search can be described as follows. Suppose $G$ is a smallest graph such that $\gamma^\infty(G)<\theta(G)$ with the maximum number of edges. Fact \ref{Fact:Bounds} implies that $\alpha(G)<\theta(G)$; moreover, by the choice of $G$ and Fact \ref{Fact:Spanning}, deleting a vertex from the graph or adding a missing edge to the graph decreases its clique covering number by $1$. Consequently, $G$ is a critical graph. Table \ref{Table:Data} shows the number of critical graphs of order $9$ or less and each of these graphs is either drawn in Figures \ref{Figure:Critical7} or \ref{Figure:Critical8}, or listed in Graph6 format in Table \ref{Appendix:Critical} in the appendix. Using the computational methods described in Section \ref{Section:Algorithms} we checked that each of these graphs has eternal domination number equal to its clique covering number.
	
	Using a similar technique, without considering only critical graphs, we found the complete set of $10$-vertex and $11$-vertex graphs with $\gamma^\infty<\theta$. Those graphs are listed in Table \ref{Appendix:GammaInftyTheta} in the appendix in Graph6 format. 

%====================================================================================================
		
		\begin{table}[h!]
		\small
		\begin{center}
			\caption{Number of critical graphs on $n$ vertices with $\gamma^\infty<\theta$.}
			\label{Table:Data}
			\begin{tabularx}{1\textwidth} { 
					| >{\centering\arraybackslash}X 
					| >{\centering\arraybackslash}X 
					| >{\centering\arraybackslash}X
					| >{\centering\arraybackslash}X
					| >{\centering\arraybackslash}X
					| >{\centering\arraybackslash}X | }
				\hline
				$n$ & Total  & $\alpha<\theta$ & Vertex-Critical \& $\alpha<\theta$ & Critical \& $\alpha<\theta$ & Critical \& $\gamma^\infty<\theta$ \\
				\hline
				5 & 21 & 1 & 1 & 1 & 0\\
				6 & 112 & 3 & 0 & 0 & 0\\
				7 & 853 & 33 & 8 & 3 & 0\\
				8 & 11117 & 498 & 7 & 4 & 0\\
				9 & 261080  & 16539 & 353 & 38 & 0\\
				10 & 11716571  & 975676 & 5159 & 290 & 1\\
				\hline
			\end{tabularx}
		\end{center}
	\end{table}

	The graph $C_5$ is the only critical graph of order $5$ with $\alpha<\theta$. We display the critical graphs of order $7$ and $8$ in figures \ref{Figure:Critical7} and \ref{Figure:Critical8}, and list the critical graphs of order $9$ with $\alpha<\theta$ in Graph6 format in the appendix (Table \ref{Appendix:Critical}).

	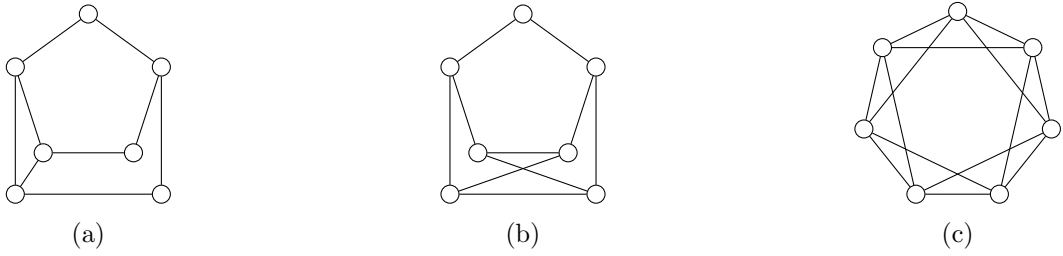
\begin{figure}[h!]
		\centering
		\begin{subfigure}[b]{0.3\textwidth}
			\centering
			\begin{tikzpicture}
				\begin{pgfonlayer}{nodelayer}
					\node [style=vertex] (0) at (0, 2.3971) {};
					\node [style=vertex] (1) at (0.9708, 1.6918) {};
					\node [style=vertex] (2) at (0.6, 0.5505) {};
					\node [style=vertex] (3) at (-0.6, 0.5505) {};
					\node [style=vertex] (4) at (-0.9708, 1.6918) {};
					\node [style=vertex] (5) at (0.9708, 0) {};
					\node [style=vertex] (6) at (-0.9708, 0) {};
				\end{pgfonlayer}
				\begin{pgfonlayer}{edgelayer}
					\draw (0) to (1);
					\draw (1) to (2);
					\draw (2) to (3);
					\draw (3) to (4);
					\draw (4) to (0);
					\draw (4) to (6);
					\draw (1) to (5);
					\draw (6) to (5);
					\draw (3) to (6);
				\end{pgfonlayer}
			\end{tikzpicture}
			
			\caption{}	
		\end{subfigure}
		\hfill
		\begin{subfigure}[b]{0.3\textwidth}
			\centering
			\begin{tikzpicture}
				\begin{pgfonlayer}{nodelayer}
					\node [style=vertex] (0) at (0, 2.3971) {};
					\node [style=vertex] (1) at (0.9708, 1.6918) {};
					\node [style=vertex] (2) at (0.6, 0.5505) {};
					\node [style=vertex] (3) at (-0.6, 0.5505) {};
					\node [style=vertex] (4) at (-0.9708, 1.6918) {};
					\node [style=vertex] (5) at (0.9708, 0) {};
					\node [style=vertex] (6) at (-0.9708, 0) {};
				\end{pgfonlayer}
				\begin{pgfonlayer}{edgelayer}
					\draw (0) to (1);
					\draw (1) to (2);
					\draw (2) to (3);
					\draw (3) to (4);
					\draw (4) to (0);
					\draw (4) to (6);
					\draw (1) to (5);
					\draw (6) to (5);
					\draw (6) to (2);
					\draw (5) to (3);
				\end{pgfonlayer}
			\end{tikzpicture}

			\caption{}
		\end{subfigure}
		\hfill
		\begin{subfigure}[b]{0.3\textwidth}
			\centering
			\begin{tikzpicture}[scale=0.555]
				\begin{pgfonlayer}{nodelayer}
					\node [style=vertex] (0) at (0, 4.381) {};
					\node [style=vertex] (1) at (1.801, 3.513) {};
					\node [style=vertex] (2) at (2.246, 1.563) {};
					\node [style=vertex] (3) at (1, 0) {};
					\node [style=vertex] (4) at (-1, 0) {};
					\node [style=vertex] (5) at (-2.246, 1.563) {};	
					\node [style=vertex] (6) at (-1.801, 3.513) {};			
				\end{pgfonlayer}
				\begin{pgfonlayer}{edgelayer}
					\draw (0) to (1);
					\draw (1) to (2);
					\draw (2) to (3);
					\draw (3) to (4);
					\draw (4) to (5);
					\draw (5) to (6);
					\draw (6) to (0);
					\draw (0) to (2);
					\draw (1) to (3);
					\draw (2) to (4);
					\draw (3) to (5);
					\draw (4) to (6);
					\draw (5) to (0);
					\draw (6) to (1);
				\end{pgfonlayer}
			\end{tikzpicture}
			\caption{}
		\end{subfigure}
		
		\caption{Critical graphs of order 7 with $\alpha<\theta$.}
		\label{Figure:Critical7}
	\end{figure}

%===== ===== ===== ===== ===== ===== ===== ===== ===== ===== ===== ===== ===== ===== ===== ===== ===== ===== ===== =====

	\begin{figure}[h!]
		\centering
		\begin{subfigure}[b]{0.24\textwidth}
			\centering
			\begin{tikzpicture}
				\begin{pgfonlayer}{nodelayer}
					\node [style=vertex] (6) at (0.9708, 1.1918) {};
					\node [style=vertex] (7) at (0.6, 0.0505) {};
					\node [style=vertex] (8) at (-0.6, 0.0505) {};
					\node [style=vertex] (9) at (-0.9708, 1.1918) {};
					\node [style=vertex] (10) at (0, 0.8763) {};
					\node [style=vertex] (11) at (0, 1.8971) {};
					\node [style=vertex] (14) at (-0.6, 2.3971) {};
					\node [style=vertex] (15) at (0.6, 2.3971) {};
				\end{pgfonlayer}
				\begin{pgfonlayer}{edgelayer}
					\draw (11) to (6);
					\draw (6) to (7);
					\draw (7) to (8);
					\draw (8) to (9);
					\draw (9) to (11);
					\draw (11) to (10);
					\draw (6) to (10);
					\draw (7) to (10);
					\draw (8) to (10);
					\draw (15) to (11);
					\draw (15) to (6);
					\draw [bend left=90, looseness=0.75] (15) to (7);
					\draw (15) to (14);
					\draw [bend right=90, looseness=0.75] (14) to (8);
				\end{pgfonlayer}
			\end{tikzpicture}
			
			\caption{}	
		\end{subfigure}
		\hfill
		\begin{subfigure}[b]{0.24\textwidth}
			\centering
			\begin{tikzpicture}[scale=0.47]
				\begin{pgfonlayer}{nodelayer}
					\node [style=vertex] (0) at (-2.5, 5) {};
					\node [style=vertex] (1) at (2.5, 0) {};
					\node [style=vertex] (2) at (1.5, 1) {};
					\node [style=vertex] (3) at (-1.5, 4) {};
					\node [style=vertex] (4) at (-2.5, 0) {};
					\node [style=vertex] (5) at (1.5, 4) {};
					\node [style=vertex] (6) at (-1.5, 1) {};
					\node [style=vertex] (7) at (2.5, 5) {};
				\end{pgfonlayer}
				\begin{pgfonlayer}{edgelayer}
					\draw (0) to (3);
					\draw (0) to (4);
					\draw (0) to (6);
					\draw (0) to (7);
					\draw (1) to (4);
					\draw (1) to (5);
					\draw (1) to (6);
					\draw (1) to (7);
					\draw (2) to (5);
					\draw (2) to (6);
					\draw (3) to (5);
					\draw (3) to (6);
					\draw (3) to (7);
					\draw (5) to (7);
				\end{pgfonlayer}
			\end{tikzpicture}
			
			\caption{}
		\end{subfigure}
		\hfill
		\begin{subfigure}[b]{0.24\textwidth}
			\centering
			\begin{tikzpicture}[scale=0.775]
				\begin{pgfonlayer}{nodelayer}
					\node [style=vertex] (0) at (-0.5, 1.75) {};
					\node [style=vertex] (1) at (0.5, 1.75) {};
					\node [style=vertex] (2) at (0, 3.0776) {};
					\node [style=vertex] (3) at (-1.618, 1.9021) {};
					\node [style=vertex] (4) at (0, 1.25) {};
					\node [style=vertex] (5) at (1.618, 1.9021) {};
					\node [style=vertex] (6) at (-0.9708, 0) {};
					\node [style=vertex] (7) at (0.9708, 0) {};
				\end{pgfonlayer}
				\begin{pgfonlayer}{edgelayer}
					\draw (0) to (2);
					\draw (2) to (1);
					\draw (1) to (0);
					\draw (0) to (3);
					\draw (3) to (6);
					\draw (1) to (5);
					\draw (5) to (7);
					\draw (7) to (4);
					\draw (4) to (6);
					\draw [bend left=15] (7) to (3);
					\draw [bend left=15] (5) to (6);
					\draw (3) to (2);
					\draw (2) to (5);
					\draw (4) to (0);
					\draw (4) to (1);
				\end{pgfonlayer}
			\end{tikzpicture}
			
			\caption{}			
		\end{subfigure}
		\hfill
		\begin{subfigure}[b]{0.24\textwidth}
			\centering
			\begin{tikzpicture}
				\begin{pgfonlayer}{nodelayer}
					\node [style=vertex] (2) at (0.9708, 0) {};
					\node [style=vertex] (3) at (-0.9708, 0) {};
					\node [style=vertex] (5) at (0, 2.3971) {};
					\node [style=vertex] (6) at (0.9708, 1.6918) {};
					\node [style=vertex] (7) at (0.6, 0.5505) {};
					\node [style=vertex] (8) at (-0.6, 0.5505) {};
					\node [style=vertex] (9) at (-0.9708, 1.6918) {};
					\node [style=vertex] (10) at (0, 1.3763) {};
				\end{pgfonlayer}
				\begin{pgfonlayer}{edgelayer}
					\draw (2) to (3);
					\draw (6) to (10);
					\draw (7) to (10);
					\draw (8) to (10);
					\draw (9) to (10);
					\draw (2) to (6);
					\draw (3) to (9);
					\draw (5) to (6);
					\draw (6) to (7);
					\draw (7) to (8);
					\draw (8) to (9);
					\draw (9) to (5);
					\draw (3) to (8);
					\draw (2) to (7);
				\end{pgfonlayer}
			\end{tikzpicture}
				
			\caption{}
			\label{Figure:G_1}
		\end{subfigure}
		
		\caption{Critical graphs of order 8 with $\alpha<\theta$.}
		\label{Figure:Critical8}
	\end{figure}
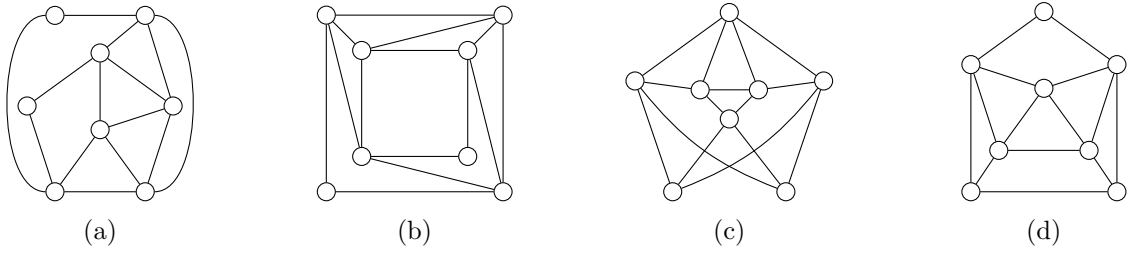

%===== ===== ===== ===== ===== ===== ===== ===== ===== ===== ===== ===== ===== ===== ===== ===== ===== ===== ===== =====

	\newpage
	\subsection{Graph Classes} \label{Section:GraphClasses}
	In this section, we focus on different classes of graphs, starting with triangle-free graphs in Sections \ref{Section:TriangleFree} and \ref{Section:MaximalTriangleFree}, before moving on to circulant graphs in Section \ref{Section:Circulant} and planar graphs in Section \ref{Section:Planar}. We briefly mention claw-free graphs in Section \ref{Section:ClawFree} and cubic graphs in Section \ref{Section:Cubic}.
	
	\subsubsection{Triangle-free graphs} \label{Section:TriangleFree}
	Erdős, Kleitman and Rothschild \cite{erdHos1976asymptotic} showed that almost all triangle-free graphs are bipartite, therefore almost all triangle-free graphs are perfect and satisfy $\gamma^\infty=\theta$. However, Goddard, Hedetniemi and Hedetniemi \cite{goddard2005eternal} showed that there exist triangle-free graphs with $\gamma^\infty<\theta$. They gave the circulant graphs $C_{18}[1, 3, 8]$ and $C_{21}[1, 3, 8]$, which they found using computer assistance, as examples. Using Algorithm $1$, we checked that any triangle-free graph of order $12$ or less has eternal domination number equal to its clique covering number. We found $13$ triangle-free graphs of order $13$ with $\gamma^\infty<\theta$ (see Table \ref{Appendix:TriangleFree}). Previously the smallest known triangle-free graph with this property has $17$ vertices and is a subgraph of the circulant graph $C_{18}[1, 3, 8]$. The following corollary describes a way to generate an infinite family of triangle-free graphs with $\gamma^\infty<\theta$.
	
	\begin{fact} \label{Fact:BowtieTriangleFree}
		Let $G$ be a triangle-free graph. Then $G \bowtie K_2$ is triangle-free.
	\end{fact}
	
	\begin{cor} \label{Corollary:TriangleFreeInfinite}
		Let $G$ be a triangle-free graph such that $\gamma^\infty(G)<\theta(G)=\lceil \frac{n}{2} \rceil$. Then, $\gamma^\infty(G \bowtie K_2) < \theta(G \bowtie K_2)$.
	\end{cor}

	\begin{proof}
		Let $G$ be a triangle-free graph such that $\gamma^\infty(G)<\theta(G)=\lceil \frac{n}{2} \rceil$. Fact \ref{Fact:BowtieTriangleFree} implies that $G \bowtie K_2$ is triangle-free; as a result, $\theta(G \bowtie K_2) \geq n$. Observe that the vertices of $G \bowtie K_2$ can be partitionned into two subsets, each of which induces a subgraph isomorphic to $G$. This means that $\gamma^\infty(G \bowtie K_2) \leq 2 \gamma^\infty(G)<2 \lceil \frac{n}{2} \rceil$; consequently, $\gamma^\infty(G \bowtie K_2) \leq n-1$.
	\end{proof}		

	Klostermeyer and Mynhardt \cite{klostermeyer2015domination} posed the following question.
	
	\begin{que} [\cite{klostermeyer2015domination}] \label{Conjecture:AlphaTheta}
		Does there exist a triangle-free graph $G$ such that $\alpha(G)=\gamma^\infty(G)<\theta(G)$?
	\end{que}

	Note that the $13$ triangle-free graphs on $13$ vertices with $\gamma^\infty<\theta$ we found also satisfy $\alpha<\gamma^\infty<\theta$. In Proposition \ref{Proposition:TriangleFree}, we describe a property of a smallest such a graph (if it exists), then we show by using computer assistance that no graph of order $14$ or less satisfies this property.
	
	\begin{prop} \label{Proposition:TriangleFree}
		Suppose there exist triangle-free graphs with $\alpha=\gamma^\infty<\theta$ and let $G$ be such a graph with minimum order. Let $X$ be a maximum independent set of $G$ and let $Y=V(G)-X$. Then $|X|=|Y|-1$.
	\end{prop}
	
	\begin{proof}
		It is clear that $|X| \geq |Y|-1$, otherwise, by Fact \ref{Fact:Spanning}, $\alpha(G-\{v\})=\gamma^\infty(G-\{v\})<\theta(G-\{v\})$ for any $v \in Y$, and $G-\{v\}$ would be a smaller triangle-free graph with $\alpha=\gamma^\infty<\theta$. So, it remains to show that $|X| \leq |Y|-1$. Suppose this is false, in other words $|X| \geq |Y|$. Let $G'$ be the spanning bipartite subgraph of $G$ obtained by deleting all edges having both endpoints in $Y$. Observe that the graph $G'$ does not contain a matching that covers all of the vertices in $Y$, otherwise, $G$ would contain a matching that matches each vertex in $Y$ to a vertex in $X$, and this would imply that $\alpha(G)=\theta(G)$ (contradiction). Consequently, by Hall's matching condition, $Y$ contains a subset $S$ such that $|N_{G'}(S)|<|S|$. Now, consider the subgraph $H$ of $G$ induced by $S \cup N_{G'}(S)$. Since there is no edge between a vertex in $H$ and a vertex in $X-V(H)$, we have $\alpha(H)=|N_{G'}(S)|$, otherwise $G$ contains an independent set of size at least $|X|+1$. Moreover, since $\gamma^\infty(G)=|X|$ we must have $\gamma^\infty(H)=|N_{G'}(S)|$: this is true because the attacker may force the $|X|$ guards to be located in $X$ and from there only attacks the vertices in $H$. In this case, only the $|N_{G'}(S)|$ guards are able to respond to that sequence of attacks on $H$. Hence, $H$ would be a smaller triangle-free graph with $\alpha=\gamma^\infty<\theta$.
	\end{proof}
	
	\begin{figure}[h]
		\centering
		\begin{subfigure}[b]{0.45\textwidth}
			\centering
			\begin{tikzpicture}
			\begin{pgfonlayer}{nodelayer}
			\node [style=vertex] (0) at (-3, 3.5) {};
			\node [style=vertex] (1) at (-1.5, 3.5) {};
			\node [style=vertex] (2) at (0, 3.5) {};
			\node [style=vertex] (3) at (1.5, 3.5) {};
			\node [style=vertex] (5) at (0.75, 0.5) {};
			\node [style=vertex] (6) at (-1.5, 0.5) {};
			\node [style=vertex] (7) at (0, 0.5) {};
			\node [style=none] (10) at (-3.5, 3.75) {};
			\node [style=none] (11) at (-3.5, 3.25) {};
			\node [style=none] (12) at (2, 3.75) {};
			\node [style=none] (13) at (2, 3.25) {};
			\node [style=none] (14) at (-3.5, 0.75) {};
			\node [style=none] (15) at (-3.5, 0.25) {};
			\node [style=none] (16) at (2, 0.25) {};
			\node [style=none] (17) at (2, 0.75) {};
			\node [style=none] (21) at (-4, 3.5) {Y};
			\node [style=none] (22) at (-4, 0.5) {X};
			\node [style=vertex] (23) at (-2.25, 3.5) {};
			\node [style=vertex] (24) at (-2.25, 0.5) {};
			\node [style=vertex] (25) at (-0.75, 3.5) {};
			\node [style=vertex] (26) at (-0.75, 0.5) {};
			\node [style=vertex] (27) at (0.75, 3.5) {};
			\end{pgfonlayer}
			\begin{pgfonlayer}{edgelayer}
			\draw (10.center) to (12.center);
			\draw (13.center) to (11.center);
			\draw [bend right=90, looseness=1.25] (10.center) to (11.center);
			\draw [bend left=90, looseness=1.25] (12.center) to (13.center);
			\draw [bend right=90, looseness=1.25] (14.center) to (15.center);
			\draw [bend left=90, looseness=1.25] (17.center) to (16.center);
			\draw (17.center) to (14.center);
			\draw (15.center) to (16.center);
			\draw (1) to (6);
			\draw (2) to (7);
			\draw (23) to (24);
			\draw (25) to (26);
			\draw (27) to (5);
			\end{pgfonlayer}
			\end{tikzpicture}
				
			\caption{}
			\label{Figure:AlphaThetaCase1}	
		\end{subfigure}
		\hfill
		\begin{subfigure}[b]{0.45\textwidth}
			\centering
			\begin{tikzpicture}
			\begin{pgfonlayer}{nodelayer}
			\node [style=vertex] (0) at (-2.25, 3) {};
			\node [style=vertex] (1) at (-0.75, 3) {};
			\node [style=vertex] (2) at (0.75, 3) {};
			\node [style=vertex] (3) at (2.25, 3) {};
			\node [style=vertex] (5) at (-2.25, 0) {};
			\node [style=vertex] (6) at (-0.75, 0) {};
			\node [style=vertex] (7) at (0.75, 0) {};
			\node [style=vertex] (8) at (2.25, 0) {};
			\node [style=none] (10) at (-3.25, 3.25) {};
			\node [style=none] (11) at (-3.25, 2.75) {};
			\node [style=none] (12) at (3.25, 3.25) {};
			\node [style=none] (13) at (3.25, 2.75) {};
			\node [style=none] (14) at (-3.25, 0.25) {};
			\node [style=none] (15) at (-3.25, -0.25) {};
			\node [style=none] (16) at (3.25, -0.25) {};
			\node [style=none] (17) at (3.25, 0.25) {};
			\node [style=none] (21) at (3.75, 3) {Y};
			\node [style=none] (22) at (3.75, 0) {X};
			\node [style=vertex] (23) at (-1.5, 3) {};
			\node [style=vertex] (24) at (-1.5, 0) {};
			\node [style=vertex] (25) at (0, 3) {};
			\node [style=vertex] (26) at (0, 0) {};
			\node [style=vertex] (27) at (1.5, 3) {};
			\node [style=vertex] (28) at (1.5, 0) {};
			\node [style=none] (29) at (-1.75, 3.15) {};
			\node [style=none] (30) at (-1.75, 2.85) {};
			\node [style=none] (31) at (1.75, 3.15) {};
			\node [style=none] (32) at (1.75, 2.85) {};
			\node [style=none] (34) at (-1, 0.15) {};
			\node [style=none] (35) at (-1, -0.15) {};
			\node [style=none] (36) at (1, 0.15) {};
			\node [style=none] (37) at (1, -0.15) {};
			\node [style=none] (38) at (0, 3.5) {S};
			\node [style=none] (39) at (0, 0.5) {$N_{G'}(S)$};
			\end{pgfonlayer}
			\begin{pgfonlayer}{edgelayer}
			\draw (10.center) to (12.center);
			\draw (13.center) to (11.center);
			\draw [bend right=90, looseness=1.25] (10.center) to (11.center);
			\draw [bend left=90, looseness=1.25] (12.center) to (13.center);
			\draw [bend right=90, looseness=1.25] (14.center) to (15.center);
			\draw [bend left=90, looseness=1.25] (17.center) to (16.center);
			\draw (17.center) to (14.center);
			\draw (15.center) to (16.center);
			\draw (0) to (5);
			\draw (1) to (6);
			\draw (2) to (7);
			\draw (3) to (8);
			\draw (25) to (26);
			\draw [bend right=90, looseness=1.25] (29.center) to (30.center);
			\draw [bend left=90, looseness=1.25] (31.center) to (32.center);
			\draw [bend right=90, looseness=1.25] (34.center) to (35.center);
			\draw [bend left=90, looseness=1.25] (36.center) to (37.center);
			\draw (29.center) to (31.center);
			\draw (32.center) to (30.center);
			\draw (34.center) to (36.center);
			\draw (35.center) to (37.center);
			\end{pgfonlayer}
			\end{tikzpicture}
									
			\caption{}
			\label{Figure:AlphaThetaCase2}	
		\end{subfigure}
		\caption{Cases considered in the proof of Proposition \ref{Proposition:TriangleFree}.}
		\label{Figure:AlphaTheta}
	\end{figure}
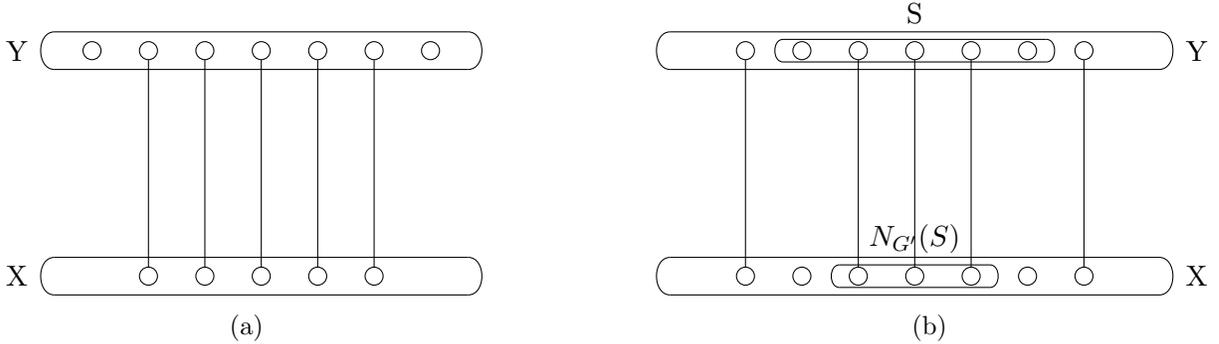
	
	Using the property stated in Proposition \ref{Proposition:TriangleFree}, to find the smallest triangle-free graph $G$ with $\alpha=\gamma^\infty<\theta$ we only need to check the triangle-free graphs of order $2k+1$ with $\alpha=k$ and $\theta=k+1$. Again, using NAUTY, we generated the set of triangle-free graphs of odd orders on fewer than $14$ vertices; then, we computed the independence number and the clique covering number of each of these graphs using the algorithms described above. If $\alpha=\frac{n-1}{2}$ and $\theta=\frac{n+1}{2}$, then we computed the eternal domination number of the graph. We performed the computation listed in Table \ref{Table:AlphaTheta} on a cluster which required approximately $218$ CPU days.
	
	\begin{obs} \label{Observation:AlphaTheta}
		There are no triangle-free graphs on $n \leq 14$ vertices with $\alpha=\gamma^\infty<\theta$.
	\end{obs}
	
	\begin{table}[h!]
		\small
		\begin{center}
			\caption{Number of connected triangle-free graphs on $n$ vertices with $\alpha=\lfloor \frac{n}{2} \rfloor$ and $\theta=\lceil \frac{n}{2} \rceil$.}
			\label{Table:AlphaTheta}
			\begin{tabularx}{1\textwidth} { 
					| >{\centering\arraybackslash}X 
					| >{\centering\arraybackslash}X 
					| >{\centering\arraybackslash}X
					| >{\centering\arraybackslash}X
					| >{\centering\arraybackslash}X | }
				\hline
				$n$ & Total & $\alpha=\lfloor \frac{n}{2} \rfloor$ & $\alpha= \lfloor \frac{n}{2} \rfloor$ \& $\theta= \lceil \frac{n}{2} \rceil$ & $\alpha=\gamma^\infty=\lfloor \frac{n}{2} \rfloor$ \& $\theta= \lceil \frac{n}{2} \rceil$ \\
				\hline
				5 & 6 & 1 & 1 & 0\\
				7 & 59 & 8 & 8 & 0\\
				9 & 1380 & 276 & 276 & 0\\
				11 & 90842 & 29660 & 29660 & 0\\
				13 & 19425052 & 9606337 & 9606334 & 0\\
				\hline
			\end{tabularx}
		\end{center}
	\end{table}
	
%===== ===== ===== ===== ===== ===== ===== ===== ===== ===== ===== ===== ===== ===== ===== ===== ===== ===== ===== =====
	
	\newpage
	\subsubsection{Maximal triangle-free graphs} \label{Section:MaximalTriangleFree}
	
	A graph $G$ is said to be \textit{maximal triangle-free} if $G$ is triangle-free and the insertion of any missing edge in $G$ creates a triangle. The maximal triangle-free graphs are considered to be an interesting family of graphs since several problems on triangle-free graphs can be studied by restricting attention to maximal triangle-free graphs. Observe that if $G$ is a triangle-free graph with $\lfloor \frac{n}{2} \rfloor = \alpha = \gamma^\infty < \theta = \lceil \frac{n}{2} \rceil$, then $G$ is a subgraph of a maximal triangle-free graph with $\gamma^\infty < \theta = \lceil \frac{n}{2} \rceil$. This follows since the clique covering number of a triangle-free graph $G$ is at least $\lceil \frac{n}{2} \rceil$ and adding any missing edge to $G$ that does not create a triangle increases neither its clique covering number nor its eternal domination number.
	
	\begin{prop} \label{Proposition:MaximalTriangleFree}
		Suppose there exist maximal triangle-free graphs with $\alpha=\gamma^\infty<\theta$ and let $G$ be such a graph with minimum order. Let $X$ be a maximum independent set of $G$ and $Y=V(G)-X$. Then $|X|=|Y|-1$.
	\end{prop}

	\begin{proof}
		Same proof as Proposition \ref{Proposition:TriangleFree}.
	\end{proof}
	
	\begin{obs} \label{Observation:MTFalphaTheta}
		There are no maximal triangle-free graphs on $n \leq 16$ vertices with $\alpha=\gamma^\infty<\theta$.
	\end{obs}

	Observation \ref{Observation:MTFalphaTheta} was verified using a similar computer search as in Observation \ref{Observation:AlphaTheta}.
		
	\begin{table}[h!]
		\small
		\begin{center}
			\caption{Number of maximal triangle-free graphs on $n$ vertices with $\alpha=\lfloor \frac{n}{2} \rfloor$ and $\theta=\lceil \frac{n}{2} \rceil$.}
			\label{Table:AlphaThetaMaximal}
			\begin{tabularx}{1\textwidth} { 
					| >{\centering\arraybackslash}X 
					| >{\centering\arraybackslash}X 
					| >{\centering\arraybackslash}X
					| >{\centering\arraybackslash}X
					| >{\centering\arraybackslash}X | }
				\hline
				$n$ & Total & $\alpha=\lfloor \frac{n}{2} \rfloor$ & $\alpha= \lfloor \frac{n}{2} \rfloor$ \& $\theta= \lceil \frac{n}{2} \rceil$ & $\alpha=\gamma^\infty=\lfloor \frac{n}{2} \rfloor$ \& $\theta= \lceil \frac{n}{2} \rceil$ \\
				\hline
				5 & 3 & 1 & 1 & 0 \\
				7 & 6 & 1 & 1 & 0 \\
				9 & 16 & 5 & 5 & 0 \\
				11 & 61 & 23 & 23 & 0 \\
				13 & 392 & 172 & 172 & 0 \\
				15 & 5036 & 1837 & 1837 & 0 \\
				17 & 164796 & 38606 & 38606 & ? \\
				\hline
			\end{tabularx}
		\end{center}
	\end{table}
	
	%===== ===== ===== ===== ===== ===== ===== ===== ===== ===== ===== ===== ===== ===== ===== ===== ===== ===== ===== =====
	
	\subsubsection{Circulant graphs} \label{Section:Circulant}
	
	Circulant graphs are another interesting family of graphs on which we study the problem. Using Algorithm $1$, we found all the circulant graphs of order $20$ or less with $\gamma^\infty<\theta$ (see Table \ref{Table:Circulants}).
	
	\begin{table}[h!]
		\small
		\begin{center}
			\caption{List of small circulant graphs with $\gamma^\infty<\theta$.}
			\label{Table:Circulants}
			\begin{tabular}{|p{0.5cm} || p{13cm}|}
				\hline
				$n$ & List of graphs \\
				\hline
				\hline
				13 & $C_{13}[1, 3, 4]$, $C_{13}[1,2,3,5]$. \\
				\hline
				14 & None \\
				\hline
				15 & $C_{15}[1, 3, 4]$. \\
				\hline
				16 & $C_{16}[1, 2, 4, 5]$, $C_{16}[1, 2, 3, 4, 6]$. \\
				\hline
				17 & $C_{17}[1, 2, 4, 8], C_{17}[1, 2, 3, 5, 6], C_{17}[1, 2, 3, 5, 8]$. \\
				\hline
				18 & $C_{18}[1, 3, 8], C_{18}[1, 2, 4, 5, 6], C_{18}[1, 2, 4, 5, 6, 9]$. \\
				\hline
				19 & $C_{19}[1, 4, 6], C_{19}[1, 3, 5, 6], C_{19}[1, 2, 3, 4, 5, 7], C_{19}[1, 2, 3, 5, 7, 8]$. \\
				\hline
				20 & $C_{20}[1, 5, 8], C_{20}[2, 5, 6], C_{20}[1, 6, 8, 9], C_{20}[1, 2, 4, 5, 6], C_{20}[1, 2, 4, 5, 7]$
				
				$C_{20}[1, 2, 5, 7, 8], C_{20}[1, 2, 3, 4, 5, 7, 8], C_{20}[1, 2, 3, 4, 6, 7, 10], C_{20}[1, 3, 4, 7, 8, 9, 10]$. \\
				\hline
			\end{tabular}
		\end{center}
	\end{table}
	
	\newpage
	\begin{prop} \label{Proposition:BowtieCirculant}
		For any integer $n$, $C_{n}[k_1, k_2, ..., k_l] \bowtie K_2 \cong C_{2n}[2k_1, 2k_2, \ldots, 2k_l, 2k_1+1, 2k_2+1, \ldots, 2k_l+1]$.
	\end{prop}

	\begin{proof}
		Let $G = C_{n}[k_1, k_2, ..., k_l]$ be a circulant graph and let $H = C_{n}[k_1, k_2, ..., k_l] \bowtie K_2$. Suppose $V(G)=\{v_1, v_2, \ldots, v_n\}$ and $V(K_2)=\{0, 1\}$. For each $i \in \{1,2,\ldots,n\}$, let $u_{2i}=(v_i,0)$ and $u_{2i+1}=(v_i,1)$. The definition of the bow tie product implies the following statements:
		\begin{itemize}
			\item The vertices $u_{2i}$ and $u_{2j}$ are adjacent in $H$ if and only if the vertices $v_i$ and $v_j$ are adjacent in $G$; that is, if and only if $i-j \equiv \pm k \pmod n$ for some $k \in \{k_1, k_2, \ldots, k_p\}$. Equivalently, $u_{2i}$ is adjacent to $u_{2j}$ if and only if $2i-2j \equiv \pm k \pmod {2n}$ for some $k \in \{2k_1, 2k_2, ..., 2k_l\}$.
			\item The vertices $u_{2i+1}$ and $u_{2j+1}$ are adjacent in $H$ if and only if the vertices $v_i$ and $v_j$ are adjacent in $G$. Therefore, $u_{2i+1}$ is adjacent to $u_{2j+1}$ if and only if $2i-2j \equiv \pm k \pmod {2n}$ for some $k \in \{2k_1, 2k_2, ..., 2k_l\}$.
			\item The vertices $u_{2i+1}$ and $u_{2j}$ are adjacent in $H$ if and only if the vertices $v_i$ and $v_j$ are adjacent in $G$; that is, if and only if $i-j \equiv \pm k \pmod n$ for some $k \in \{k_1, k_2, \ldots, k_p\}$. Equivalently, $u_{2i+1}$ is adjacent to $u_{2j}$ if and only if $(2i+1)-2j \equiv \pm k \pmod {2n}$ for some $k \in \{2k_1+1, 2k_2+1, ..., 2k_l+1\}$.
		\end{itemize}
	Thus, for each $i, j \in \{0,1,2,\ldots,2n-1\}$, the vertices $u_i$ and $u_j$ are adjacent if and only if $i-j \equiv \pm k \pmod {2n}$ for some $k \in \{2k_1, 2k_2, \ldots, 2k_l, 2k_1+1, 2k_2+1, \ldots, 2k_l+1\}$.
	\end{proof}
	
	\begin{cor}
		There exist infinitely many circulant graphs with $\gamma^\infty<\theta$.
	\end{cor}

	\begin{proof}
		This follows from Fact \ref{Fact:BowtieTriangleFree}, Corollary \ref{Corollary:TriangleFreeInfinite} and Proposition \ref{Proposition:BowtieCirculant}.
	\end{proof}

%===== ===== ===== ===== ===== ===== ===== ===== ===== ===== ===== ===== ===== ===== ===== ===== ===== ===== ===== =====

	\subsubsection{Planar graphs} \label{Section:Planar}
	
	Anderson, Barrientos, Brigham, Carrington, Vitray and Yellen \cite{anderson2007maximum} showed that any outerplanar graph has eternal domination number equal to its clique covering number. Moreover, no planar graph where $\gamma^\infty<\theta$ is known. This suggest that it might be true for general planar graphs and motivates the following question of Klostermeyer and Mynhardt.
		
	\begin{que} [\cite{klostermeyer2016protecting}] \label{Question:Planar}
		Is it true that $\gamma^\infty(G)=\theta(G)$ if $G$ is planar?
	\end{que}

	We first describe some properties of a smallest planar graph with $\gamma^\infty<\theta$, if it exists.
	
	\begin{prop}
		Let $\mathcal{F}$ be a family of graphs satisfying a hereditary property $\mathcal{P}$; i.e. if $G \in \mathcal{F}$ and $G$ satisfies property $\mathcal{P}$, then $H$ satisfies property $\mathcal{P}$ for any subgraph $H$ of $G$. Suppose $\mathcal{F}$ contains graphs $G$ such that $\gamma^\infty(G)<\theta(G)$. Then, the smallest such graph $G$ is a $2$-connected vertex critical graph.
	\end{prop}

	\begin{proof}
		It is clear that $G$ is a vertex-critical graph; otherwise, by Fact \ref{Fact:Spanning}, $G$ contains a proper induced subgraph which is also in $\mathcal{F}$ with $\gamma^\infty<\theta$ (contradiction). Suppose $G$ is not $2$-connected and let $v$ be a cut vertex of $G$, that is a vertex such that $G-\{v\}$ has $k$ components $\{G_1, G_2, \ldots, G_k\}$ where $k\geq2$. By the minimality of $G$, $\gamma^\infty(G_i)=\theta(G_i)$ for each $i \in \{1,2,\ldots,k\}$. Since $\gamma^\infty(G) \geq \sum_{i=1}^k \gamma^\infty(G_i)$, $\theta(G) \leq \sum_{i=1}^k \theta(G_i)+1$ and $\gamma^\infty(G)<\theta(G)$, we conclude that $\gamma^\infty(G) = \sum_{i=1}^k \gamma^\infty(G_i)$ and $\theta(G) = \sum_{i=1}^k \theta(G_i)+1$. This implies that there exists $i \in \{1,2,\ldots,k\}$ such that $\gamma^\infty(G_i)=\gamma^\infty(G_i+\{v\})$ and $\theta(G_i+\{v\})=\theta(G_i)+1$. In this case, $G_i+\{v\}$ is a smaller graph in $\mathcal{F}$ with $\gamma^\infty<\theta$, which is a contradiction.
	\end{proof}
	
	Since planarity is a hereditary property, the smallest planar graph with $\gamma^\infty<\theta$, if it exists, is a $2$-connected vertex-critical graph. Observe that none of the graphs on $11$ vertices or less satisfying $\gamma^\infty<\theta$ we found (Table \ref{Appendix:GammaInftyTheta}) are planar, hence we have the following observation.

	\begin{obs} \label{Observation:Planar}
		There are no planar graphs on $11$ vertices or less with $\gamma^\infty<\theta$.
	\end{obs}

	We also considered planar graphs of higher orders ($12$ and $13$) in our search and we used plantri \cite{brinkmann2007fast, brinkmann2007plantri} (version $5.2$) to generate them. Due to the limitations of plantri, we only considered $3$-connected planar graphs and obtained the following observation.
	
	\begin{obs} \label{Observation:3ConnectedPlanar}
		There are no $3$-connected planar graphs on $13$ vertices or less with $\gamma^\infty<\theta$.
	\end{obs}

	\begin{table}[h!]
		\small
		\begin{center}
			\caption{Number of $3$-connected planar graphs on $n$ vertices.}
			\label{Table:3ConnectedPlanar}
			\begin{tabularx}{1\textwidth} { 
					| >{\centering\arraybackslash}X 
					| >{\centering\arraybackslash}X 
					| >{\centering\arraybackslash}X
					| >{\centering\arraybackslash}X
					| >{\centering\arraybackslash}X | }
				\hline
				$n$ & Total & $\alpha<\theta$ & Vertex-Critical \& $\alpha<\theta$ & Vertex-Critical \& $\gamma^\infty<\theta$ \\
				\hline
				\hline
				10 & 32300 & 2773 & 14 & 0 \\
				11 & 440564 & 25771 & 74 & 0 \\
				12 & 6384634 & 745440 & 878 & 0 \\
				13 & 96262938 & 6774391 & 2475 & 0 \\
				\hline
			\end{tabularx}
		\end{center}
	\end{table}

%===== ===== ===== ===== ===== ===== ===== ===== ===== ===== ===== ===== ===== ===== ===== ===== ===== ===== ===== =====

	\subsubsection{Claw-free graphs} \label{Section:ClawFree}
	
	A \emph{claw-free graph} is a graph that does not contain a claw as an induced subgraph. In particular, any graph with independence number $2$ is claw-free. The complements of the Mycielski graphs $M_k$ described in Section \ref{Section:PreliminaryResults} have independence number $2$, eternal domination number $3$ and clique covering number $k$ for $k \geq 4$. Consequently, there are infinitely many claw-free graphs with $\gamma^\infty<\theta$.

%===== ===== ===== ===== ===== ===== ===== ===== ===== ===== ===== ===== ===== ===== ===== ===== ===== ===== ===== =====

	\subsubsection{Cubic graphs} \label{Section:Cubic}
	
	A \emph{cubic graph} is a graph in which all vertices have degree $3$. Using NAUTY along with Algorithm $1$, we generated the set of cubic graphs on fewer than $18$ vertices and obtained the following observation.
	
	\begin{obs}
		There are no cubic graph on $n \leq 16$ vertices with $\gamma^\infty<\theta$.
	\end{obs}
	
	\begin{table}[h!]
		\small
		\begin{center}
			\caption{Number of connected cubic graphs on $n$ vertices.}
			\label{Table:Cubic}
			\begin{tabularx}{0.8\textwidth} { 
					| >{\centering\arraybackslash}X 
					| >{\centering\arraybackslash}X 
					| >{\centering\arraybackslash}X
					| >{\centering\arraybackslash}X | }
				\hline
				$n$ & Total & $\alpha<\theta$ & $\gamma^\infty<\theta$ \\
				\hline
				\hline
				4 & 1 & 0 & 0 \\
				6 & 2 & 0 & 0 \\
				8 & 5 & 2 & 0 \\
				10 & 19 & 9 & 0 \\
				12 & 85 & 46 & 0 \\
				14 & 509 & 320 & 0 \\
				16 & 4060 & 2888 & 0 \\
				\hline
			\end{tabularx}
		\end{center}
	\end{table}
	
%===== ===== ===== ===== ===== ===== ===== ===== ===== ===== ===== ===== ===== ===== ===== ===== ===== ===== ===== =====
	
	\subsection{Domination, eternal domination and clique covering} \label{Section:DominationEternalDomination}
	
	As we have seen above, the guards must always be located on the vertices of a dominating set in a graph $G$ in order to defend a sequence of attacks on $G$. With this in mind, many researchers showed interest in characterizing graphs for which the domination number is equal to the eternal domination number. Klostermeyer and Mynhardt \cite{klostermeyer2020eternal} posed the following question.
	
	\begin{que} [\cite{klostermeyer2020eternal}] \label{Question:GammaSetsGeneral}
		Let $G$ be a graph with $\gamma(G)=\gamma^\infty(G)$. Is any minimum dominating set of $G$ an eternal dominating set of $G$?
	\end{que}
	
	The answer to Question \ref{Question:GammaSetsGeneral} is no. We state this in the following proposition.
	
	\begin{prop}
		For any integer $k \geq 2$, there exists a graph $G$ such that $\gamma(G)=\gamma^\infty(G)=k$ having minimum dominating sets which are not eternal dominating sets.	
	\end{prop}
	
	\begin{proof}
	Consider the graph $G_1$ in Figure \ref{Figure:GammaInfinitySets}. It can be easily checked that $\gamma(G_1)=\gamma^\infty(G_1)=\theta(G_1)=2$. Moreover, the set $\{u_1, u_4\}$ is a dominating set of $G$; however, if the guards are located on the vertices $u_1$ and $u_4$, then any response to an attack on the vertex $u_0$ leaves either the vertex $u_2$ or the vertex $u_3$ undominated on the next turn. Figure \ref{Figure:GeneralGammaInfinitySets} shows a generalisation of the graph in Figure \ref{Figure:GammaInfinitySets} where $\gamma(G_2)=\gamma^\infty(G_2)=k$. The set $\{u_1, u_4\} \cup \{v_0, v_1, \ldots, v_{k-3}\}$ is a dominating set of size $k$, but any response to an attack on vertex $u_0$ leaves one of the vertices $u_2, u_3$ or $w_0$ undominated on the next turn.
	\end{proof}
	
	\begin{figure}[h!]
		\centering
		\begin{subfigure}[b]{0.45\textwidth}
			\centering
			\begin{tikzpicture}
			\begin{pgfonlayer}{nodelayer}
				\node [style=vertex] (0) at (0, 3) {$u_0$};
				\node [style=vertex] (1) at (-1, 2) {$u_1$};
				\node [style=vertex] (2) at (-1, 0) {$u_2$};
				\node [style=vertex] (3) at (1, 0) {$u_3$};
				\node [style=vertex] (4) at (1, 2) {$u_4$};
			\end{pgfonlayer}
			\begin{pgfonlayer}{edgelayer}
				\draw (1) to (2);
				\draw (4) to (3);
				\draw (0) to (1);
				\draw (3) to (2);
				\draw (1) to (4);
				\draw (0) to (4);
			\end{pgfonlayer}
		\end{tikzpicture}
												
		\caption{A graph $G_1$ with $\gamma(G_1)=\gamma^\infty(G_1)=2$ having more minimum dominating sets than eternal dominating sets.}
		\label{Figure:GammaInfinitySets}	
		\end{subfigure}
		\hfill
		\begin{subfigure}[b]{0.45\textwidth}
			\centering
\begin{tikzpicture}
	\begin{pgfonlayer}{nodelayer}
		\node [style=vertex] (0) at (-2.5, 3) {$u_0$};
		\node [style=vertex] (1) at (-3.5, 2) {$u_1$};
		\node [style=vertex] (2) at (-3.5, 0) {$u_2$};
		\node [style=vertex] (3) at (-1.5, 0) {$u_3$};
		\node [style=vertex] (4) at (-1.5, 2) {$u_4$};
		\node [style=vertex] (5) at (-0.5, 3) {$v_0$};
		\node [style=vertex] (6) at (-0.5, 2) {$w_0$};
		\node [style=vertex] (7) at (0.5, 3) {$v_1$};
		\node [style=vertex] (8) at (0.5, 2) {$w_1$};
		\node [style=vertex] (9) at (1.5, 3) {$v_2$};
		\node [style=vertex] (10) at (1.5, 2) {$w_2$};
		\node [style=vertex] (11) at (3.5, 3) {\tiny $v_{k-3}$};
		\node [style=vertex] (12) at (3.5, 2) {\tiny $w_{k-3}$};
		\node [style=none] (13) at (2, 3) {};
		\node [style=none] (14) at (3, 3) {};
		\node [style=none] (15) at (2, 2) {};
		\node [style=none] (16) at (3, 2) {};
		\node [style=none] (17) at (2.5, 2) {$...$};
		\node [style=none] (18) at (2.5, 3) {$...$};
	\end{pgfonlayer}
	\begin{pgfonlayer}{edgelayer}
		\draw (1) to (2);
		\draw (4) to (3);
		\draw (0) to (1);
		\draw (3) to (2);
		\draw (1) to (4);
		\draw (0) to (4);
		\draw (0) to (5);
		\draw (5) to (6);
		\draw (5) to (7);
		\draw (7) to (8);
		\draw (7) to (9);
		\draw (9) to (10);
		\draw (9) to (13.center);
		\draw (11) to (14.center);
		\draw (11) to (12);
	\end{pgfonlayer}
\end{tikzpicture}
								
		\caption{A graph $G_2$ with $\gamma(G_2)=\gamma^\infty(G_2)=k$ having more minimum dominating sets than eternal dominating sets.}
		\label{Figure:GeneralGammaInfinitySets}
		\end{subfigure}
		
		\caption{Two graphs having more minimum dominating sets than eternal dominating sets.}
		\label{Figure:GammaSetsGammaInfinitySets}
	\end{figure}
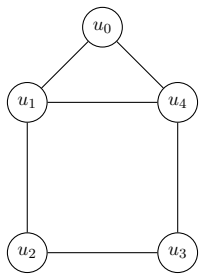
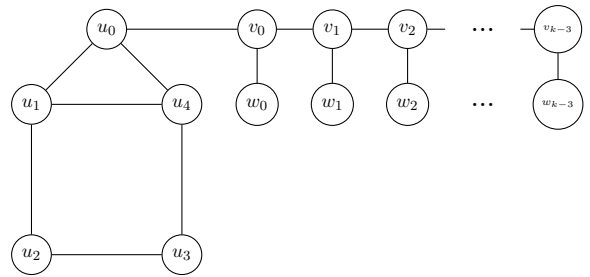
	
	A question of Klostermeyer and MacGillivray, which was later stated as a conjecture by Klostermeyer and Mynhardt, is the following.
	
	\begin{con}[$\gamma-\theta$ Conjecture \cite{klostermeyer2016protecting, klostermeyer2020eternal}] \label{Conjecture:GammaTheta}
		For any graph $G$, $\gamma(G)=\gamma^\infty(G) \iff \gamma(G)=\theta(G)$.
	\end{con}
	
	Our computer search yields the following observation.
	
	\begin{obs}
		There are no counterexample to the $\gamma-\theta$ conjecture of order $n \leq 11$.
	\end{obs}

	\begin{table}[h!]
		\small
		\begin{center}
			\caption{Number of connected graphs on $n$ vertices with $\gamma=\alpha$, $\gamma=\gamma^\infty$ and $\gamma=\theta$.}
			\label{Table:GammaTheta}
			\begin{tabularx}{0.8\textwidth} { 
					| >{\centering\arraybackslash}X 
					| >{\centering\arraybackslash}X 
					| >{\centering\arraybackslash}X
					| >{\centering\arraybackslash}X
					| >{\centering\arraybackslash}X | }
				\hline
				$n$ & Total & $\gamma=\alpha$ & $\gamma=\gamma^\infty$ & $\gamma=\gamma^\infty=\theta$ \\
				\hline
				\hline
				5 & 21 & 6 & 5 & 5 \\
				6 & 112 & 24 & 22 & 22 \\
				7 & 853 & 88 & 67 & 67 \\
				8 & 11117 & 524 & 358 & 358 \\
				9 & 261080 & 4515 & 2265 & 2265 \\
				10 & 11716571 & 73515 & 23394 & 23394 \\
				11 & 1006700565 & 2324209 & 396755 & 396755 \\
				\hline
			\end{tabularx}
		\end{center}
	\end{table}

	The computation in Table \ref{Table:GammaTheta} was performed on a cluster which required approximately $85$ CPU days. We found $56$ graphs (listed in Table \ref{Appendix:GammaInftyTheta} in the appendix) with $\gamma^\infty<\theta$ among which there are $55$ graphs with $\alpha=\gamma^\infty<\theta$ and none with $\gamma=\gamma^\infty<\theta$. 

%===== ===== ===== ===== ===== ===== ===== ===== ===== ===== ===== ===== ===== ===== ===== ===== ===== ===== ===== =====
	
	\newpage
	\section{Acknowledgements}
	All the computations in this paper were run on a PowerEdge R$7425$ server equipped with two AMD EPYC processors and totalling $64$ threaded CPU cores. This machine was purchased using NSERC funding by Wendy Myrvold. We thank Wendy for her help with the computations. \newline
	
	\noindent We acknowledge the support of the Natural Sciences and Engineering Research Council of Canada (NSERC), PIN $04459, 253271$. \newline
	
	\noindent Cette recherche a \'{e}t\'{e} financ\'{e}e par le Conseil de recherches en sciences naturelles et en g\'{e}nie du Canada (CRSNG), PIN $04459, 253271$.

	\begin{center}
		\includegraphics[width=2.5cm]{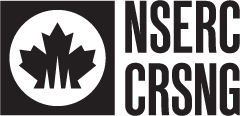}
	\end{center}

%===== ===== ===== ===== ===== ===== ===== ===== ===== ===== ===== ===== ===== ===== ===== ===== ===== ===== ===== =====

%===== ===== ===== ===== ===== ===== ===== ===== ===== ===== ===== ===== ===== ===== ===== ===== ===== ===== ===== =====
	
	\newpage
	\section*{Appendix}
	
	All the graphs in this appendix are listed in Graph6 format.
	
	\begin{table}[h!]
		\small
		\begin{center}
			\caption{List of critical graphs with $\alpha<\theta$.}
			\label{Appendix:Critical}
			\begin{tabular}{|p{0.5cm} | p{14.5cm}|}
				\hline
				$n$ & List of graphs \\
				\hline
				\hline
				$5$ & 
				\verb|DUW| \\
				\hline
				$7$ & 
				\verb|FCptO	FCxv?	FUzro| \\
				\hline
				$8$ & 
				\verb|GCrbuW	GCpveg	GCxvcw	GEhuSw| \\
				\hline
				$9$ & 
				\verb|H?bBVbQ	H?bBTjQ		H?bBThY		H?bBTiU		H?bDJqU		H?bF`xw		H?bDjpw		H?bfBqU		H?bbVaU		H?bbV_]		H?bvbro		H?rF`zo		H?q`qjo		H?o~fbo		HCQ`faw		HCQeJaL		HCrfRym		HCrbvW}		HCrUrze		HCpvfim		HCpvexy		HCpvVW}		HCpvUzq		HCpvUzU		HCpvRym		HCpunrM		HCpunp]		HCrJvi]		HCzbvZe		HCxvfri		HCxvfpy		HCxvez[		HCxvezM		HCvdjrM		HEjfaxu		HEhuTxm		HEhvTy{		HUzvvx}| \\
				\hline
			\end{tabular}
		\end{center}
	\end{table}

	\begin{table}[h!]
		\small
		\begin{center}
			\caption{List of connected graphs with $\gamma^\infty<\theta$.}
			\label{Appendix:GammaInftyTheta}
			\begin{tabular}{|p{0.5cm} | p{14.5cm}|}
				\hline
				$n$ & List of graphs \\
				\hline
				\hline
				$10$ & \verb|IEhbtj{ro		IEhbtn{ro| \\
				\hline
				$11$ & 
				\vspace{-18pt}				
				\begin{verbatim}
					JQyurj]yt|?		JEhbtj{rvu?		JEhbtj{rvx?		JEhbtj{rvf?		JEhbtj{ruv?		JEhbtj{rvT_		JEhbtj{rtt_		JEhbtj{rrt_		JEhbtj{rv}?		JEhbtj{rv|?		JEhbtj{rvv?		JEhbtj{rv^?		JEhbtj{rvx_		JEhbtj{rvt_		JEhbtj{rvl_		JEhbtj{rv\_		JEhbtj{rvf_		JEhbtj{rtv_		JEhbtj{rv~?		JEhbtj{rv|_		JEhbtj{rvv_		JEhbtj{rv~_		JEhbtnm~E|?		JEhbtnm~FZ?		JEhbtnm~D]_		JEhbtnm~@}_		JEhbtnN~Fu?		JEhbtnN~Bv?		JEhbtnN~Fw_		JEhbtnN~Fe_		JEhbtnN~Eu_		JEhbtnN~Bu_		JEhbtnN~Ef_		JEhbtnN~F}?		JEhbtnN~Fv?		JEhbtnN~F{_		JEhbtnN~Fu_		JEhbtnN~F]_		JEhbtnN~Ff_		JEhbtnN~Ev_		JEhbtnN~Bv_		JEhbtnN~F~?		JEhbtnN~F}_		JEhbtnN~Fv_		JEhbtnN~F~_		JEhvUtn~D{_		JEhutz{xr\_		JEhru^u~Fj?		JEhru]v~Et_		JEjbtnN~Dm_		JEnfbz\zbv?		JCXetqu|uz?		JCXetq}vVm?		JCXetq}|uz?		
				\end{verbatim} \\
				\hline
			\end{tabular}
		\end{center}
	\end{table}
	
	\begin{table}[h!]
		\small
		\begin{center}
			\caption{List of triangle-free graphs with $\gamma^\infty<\theta$.}
			\label{Appendix:TriangleFree}
			\begin{tabular}{|p{0.5cm} | p{14.5cm}|}
				\hline
				$n$ & List of graphs \\
				\hline
				\hline
				$13$ & 
				\verb|L?`@F?M]DgYOFo L?`DAboU`w@{hS L?`@?boNAsLGBw L?`@?boNAsOwYO L?`@?boNAs@{os| 
				\verb|L?`@Cbod`w@{YS L?BDB?{{AsRGBs L?`@C`wl?{DgQs L?BDB?{Ucq^?Fo L?`@F@wlEcBoBo| 
				\verb|L?`@F@wlEcBoFo L?`@F@wlEcBgFo L?`@F?kQ_}YSlC| \\
				\hline
			\end{tabular}
		\end{center}
	\end{table}

%===== ===== ===== ===== ===== ===== ===== ===== ===== ===== ===== ===== ===== ===== ===== ===== ===== ===== ===== =====

	\newpage
	\begin{table}[h!]
		\small
		\begin{center}
			\caption{List of maximal triangle-free graphs with $\gamma^\infty<\theta$.}
			\label{Appendix:MaximalTriangleFree}
			\begin{tabular}{|p{0.5cm} | p{14.5cm}|}
				\hline
				$n$ & List of graphs \\
				\hline
				\hline
				$13$ & 
				\verb|L?`DAboU`w@{hS	L?BDB?{{AsRGBs	L?BDB?{Ucq^?Fo	L?`@F?kQ_}YSlC	L?`@?boNAsOwYO| \\
				\hline
				$14$ & 
				\verb|M?`@F?kQckBwsglC? M?`@?boNAs@{oshQ?| \\
				\hline
				$15$ & 
				\vspace{-18pt}
				\begin{verbatim}
				N??EFBOK_}DsrC~?^_? N??EFBOK_ZbwJgrC^_? N?AAF@SBo}TSm_{CN_? N?AA@aoqTsVO^?BwGl_
				N?AADB_HdsN_VOE[B{? N?AADB_HeobMm_VOB{? N?AADB_HaqbM^?m_B{? N?ABB?WwGNN_eoFs^@?
				N?ABAaQFbgTG}?JWPX? N??CBBOk@\HqfO^_@}? N?`@F?kQcKBwsglCN@? N??ED?qw@{JgRW~?By?
				N??ED?qwAaHkjA~?By? N??ED?qwAaHkJa~?By? N??ED?qwAaPdRW~?By? N??ED?qwAyJgRW~?By?
				N??ED?qwEXHkN_~?By? N??ED?qwEXHkVOFs^_? N??ED?qwAZHkN_~?By? N??CB@OeOmCuiI~?No?
				N?AAD@WUDD_}N_VQ^?_ N??CFBOBw}Ds^_~?No? N??EDBo{@{JgFoB{^_? N?AA@BOHn_N_m_uS@|?
				N??CFBOK_}BwJgrCNo? N??CB@OKCyTcrC^_@}? N??ED@_MdoRC~?FwGj_ N??EDbGl?]EsLg~?HIo
				N?AA@B_@zwVOuOFw^A? N?AAF?e{EpN_m_Fo@y? N?AA@Bo{@{JgeoB{@}? N??ED@_Nfo^?VOfG?|_
				N??ED@_NayRc~?Fw^_? N??ED@_NayRcN_~?B{? N?AA@Boy?^FoVOeoAy? N?AADb_RO^Bs^?m_UQ?
				N?AADb_RTc~?N_BwBs? N?AA@b_HaYPWFguCN_? N??CB?Xs?]cufOxGNo? N?AA@BoZ?^QYm_VOFo?
				N??CABoyBKRHBwDsVg? N?AA@?O{BwVOuOyG@}? N?AA@?O{BwRWigRW@}? N?AA@?O{@kJgZGaw@N?
				N?AA@?O{@kJgigB{XL? N?AA@?O{@{JgeoyG@}? N?AA@?O{@{JgZGaw@}? N?AA@?O{@{XKVOig@}?
				N?AA@?O{?}DsRWaw@}? N??E@b_sEYBsJc~?@}? N?AAD@O{Ai^?N_ii@}? N?AAD@O{AiN_N_B{TT?
				N?AAD@O{AyN_m_ig@}? N??ED?WHiJJIMaxG^_? N?ABA_oebgTGqGBwP\? N?ABA_oedQXC^?BwGf?
				N?AA@b_qQYAu^?m_@{? N?AA@b_qO^Bs^?m_Eq? N?AA@b_qOl`{^?m_^?? N?AA@b_qOl`{^?m_N_?
				N??EDaKIaYMCN_tCET? N?AA@BoJ_^YI^?eoVO? N?AA@?WwGNN_m_uON_? N?AA@?WwGNN_m_uOBy?
				N?AA@?WwGNN_m_uO^@? N?AA@b_RRgBsN_aqUQ? N?AA@b_RSUEoFg}?XG_ N?AA@b_RSUUQ\AFg^?? 
				N??CFBOKeWbwJg^_^_? N??CFBOKeW`mFoJgNo? N??ED?qN_}HkrG~?^_? N??ED?qN_}HkrGVO^_?
				N??ED?qSROQ`FobK^_? N??ED?qSO}HkJc~?^_? N??ED?qSO}HkJcrG^_? N??ED?qSO}HkJcbK^_?
				N??ED@OAw^FoVOjG]a? N?AA@B_{BwVOuOFw@}? N?AA@B_{@{JgeoB{B{? N??EDaKRR`]?N_Br^_?
				N??CB@OyDPRgBwDuNo? N?AAF@Sy?mtSm_FoN_? N?AA@AOwFoRWJgRWBy? N?AA@AOwDsZGZGFs@}?
				N?AA@ao{@kGkqQB{VO_ N?AAD@Wgno^?N_VS?|_ N??CEBoy?^Ay~?NoNo? N??CBBOyBKRHDsJg^_?
				N?AADBOQRwDsm_FoEq? N?AAD?oqN_VOLcVO@|? N?AAD?oqJgVOBwVOPT? N??CFBOQ_}BwJglC^_?
				N??CFBOQ_}DslC^_^_? N??ED?qEqyHkxCN_^_? N??ED?qEo}HkxC~?^_? N??ED?qEo}WeRW~?^_?
				N??ED?qEqZHkN_xC^_? N?AA@B_}BwFoVOeoB{? N?AA@B_}@{FoVOeoB{? N?AA@B_}@{JgeoFwB{?
				N?AA@AO{BwRWJgRW@}? N?AA@AO{DsZGFoZG@}? N?AA@AO{CuFoVOZG@}? N?AA@AO{BHVOVOeq@}?
				N?AA@AO{BHFoVOB{RX? N?AA@BO{FoJgeoJg@}? N?AA@BO{@{JgeoJg@}? N??CFBO{?}DsB{~?No?
				N?`@?boNAs[G`oBwBo_ N?`@?boNAsLG`o`o@{? N?`@?boNAsLG`o`oN?W N?`@?boNAsSgooBo?Fw
				N?AA@B_\BwVOuOA{B{? N?AA@B_\BwVOasuOB{? N?AA@B_\DSPYTQuOB{? N?AA@B_\@[YH}?asB{?
				N?AA@B_\@[YHasVOB{? N?AA@B_\CU~?N_VOB{? N?AA@B_\CUvON_VOB{? N?AA@B_\CUvOuON_B{?
				N?AA@B_\CUvOuOFwB{? N?AA@B_\CUfoVOeoB{? N??ED@o]@sJGBw~?Of_ N?AA@_owBwVOFcqW@}?
				N?AA@BO}FoJgeoJg?~? N?AA@_goW]As}?^?VO? N?AA@_goW]As^?m_N_? N?AA@_goW]As^?m_Fq?
				N?AA@_goW]As^?m_]@? N?AA@_goW]As^?m_YD? N?AA@_goW]As^?m_[@_ N?AA@b@BpLZ@^?m_Fs?
				N?AA@b@Bo\N_m_{EFs? N?AA@b@Bv_n_m_NgFs? N?AA@BOX@\@]VO}?RY? N?AAF@ScfoTSZCFoBw?
				N?AAF@Sc`kBwigZC@N? N??ED?ZDp{[cVOyC^_? N?AA@_o{FoFoRWaw@}? N?AA@_o{@kGkB{qYVO_
				N?AA@_o{AJN_qWB{VP? N??ED?ZFp{[cVO~?^_? N??FEaK[N_BoJO~?@z? N?AA@?WeF@Cueo^?Gm?
				N??CB@OeRcCuFobY^_? N?AADB_\FoVOVOA{B{? N?AADB_\BwVOVOA{B{? N?AA@B_HbwVOVOFwZI?
				N?AA@B_HbwVOFwFwZI? N?AA@B_HbwVOE[FwZI? N?AA@B_HaI^?uSmaB{? N?AA@B_HaIN_FwuSVP?
				N?AA@B_HaIN_EMeoN_? N?AA@B_HaIN_EMeoRW? N?AA@B_HaIN_EMeoB{? N?AA@B_HaIOW}?FwBF_
				N?AA@B_HaIRWFwuSVP? N?AA@B_HaIB{uSmaB{? N?AA@B_HaiTP}?uSB{? N?AA@B_HaiTP^?FwZI?
				N?AA@B_HaiTPm_FwZI? N?AA@B_HayRW}?FwZI? N?AA@B_HayRW^?FwZI? N?AA@B_H_rjgm_FwZI?
				N?AADAo{AJFoVOB{VP? N??ED@_{@{FoVOfG@}? N??ED@_{@{JgfGB{^_? N??ED@_{@{JgfGB{B{?
				N?AA@_o}@{^?RWaw?}_ N??CFBOQdguaFoJg^_? N?AADB_kdsN_RSFwB{? N?AADB_k_NHi^?Fw^@?
				N?ABAaQT?]N_}?iSDM?
				\end{verbatim} \\ 
				\hline
			\end{tabular}
		\end{center}
	\end{table}

\newpage
	\begin{thebibliography}{100}
	
	\bibitem{alon1992probabilistic} \textsc{N. Alon and J. H. Spencer}. The probabilistic method. With an appendix by Paul Erdős. \textit{Wiley-Interscience Series in Discrete Mathematics and Optimization}, John Wiley \& Sons, (1992).
 
	\bibitem{anderson2007maximum} \textsc{M. Anderson, C. Barrientos, R. C. Brigham, J. R. Carrington, R. P. Vitray and J. Yellen}. Maximum-demand graphs for eternal security. \textit{J. Combin. Math. Combin. Comput.} 61 (2007), 111–127.
 
 	\bibitem{brinkmann2007plantri} \textsc{G. Brinkmann. and B. D. McKay}. The program \textbf{plantri}. Available at \url{https://users.cecs.anu.edu.au/~bdm/plantri/}
 	
 	\bibitem{brinkmann2007fast} \textsc{G. Brinkmann and B. D. McKay}. Fast generation of planar graphs. \textit{MATCH Commun. Math. Comput. Chem.} 58 (2007), no. 2, 323–357. 
 	
 	\bibitem{burger2004infinite} \textsc{A. P. Burger, E. J. Cockayne, W. R. Grundlingh, C. M. Mynhardt, J. H. van Vuuren and W. Winterbach}. Infinite order domination in graphs. \textit{J. Combin. Math. Combin. Comput.} 50 (2004), 179–194. 
 	
 	\bibitem{chvatal1974minimality} \textsc{V. Chv{\'a}tal}. The minimality of the Mycielski graph. \textit{Graphs and Combinatorics (Proc. Capital Conf., George Washington Univ., Washington, D.C., 1973), Lecture Notes in Math.} 406 (1974), 243–246, Springer, Berlin.
 	
 	\bibitem{descartes1947three} \textsc{B. Descartes}. A three colour problem. \textit{Eureka} 9 (1947), 21.
 	
 	\bibitem{descartes1954solution} \textsc{B. Descartes}. Solution to advanced problem no. 4526. \textit{Amer. Math. Monthly} 61 (1954), 352.
 	
 	\bibitem{erdHos1976asymptotic} \textsc{P. Erd{\H{o}}s, D. J. Kleitman and B. L. Rothschild}. Asymptotic enumeration of $K_n$-free graphs. \textit{Colloquio Internazionale sulle Teorie Combinatorie (Rome, 1973)}, 2 (1976), 19–27. 
 	
 	\bibitem{goddard2005eternal} \textsc{W. Goddard, S. M. Hedetniemi and S. T. Hedetniemi}. Eternal security in graphs. \textit{J. Combin. Math. Combin. Comput.} 52 (2005), 169–180.
 	
 	\bibitem{grtschel1988geometric} \textsc{M. Grötschel, L. Lovász and A. Schrijver}. Geometric algorithms and combinatorial optimization. Springer-Verlag, (1988).
 	
	\bibitem{klostermeyer2007complexity} \textsc{W. F. Klostermeyer}. Complexity of eternal security. \textit{J. Combin. Math. Combin. Comput.} 61 (2007), 135–140. 	
 	
 	\bibitem{klostermeyer2016dynamic} \textsc{W. F. Klostermeyer, M. Lawrence and G. MacGillivray}. Dynamic dominating sets: the eviction model for external domination. \textit{J. Combin. Math. Combin. Comput.} 97 (2016), 247–269. 
 	
 	\bibitem{klostermeyer2007eternal} \textsc{W. F. Klostermeyer and G. MacGillivray}. Eternal security in graphs of fixed independence number. \textit{J. Combin. Math. Combin. Comput.} 63 (2007), 97–101. 
 	
 	\bibitem{klostermeyer2009eternal} \textsc{W. F. Klostermeyer and G. MacGillivray}. Eternal dominating sets in graphs. \textit{J. Combin. Math. Combin. Comput.} 68 (2009), 97–111. 
 	
 	\bibitem{klostermeyer2015domination} \textsc{W. F. Klostermeyer and C. M. Mynhardt}. Domination, eternal domination and clique covering. \textit{Discuss. Math. Graph Theory} 35 (2015), no. 2, 283–300. 
 	
 	\bibitem{klostermeyer2016protecting} \textsc{W. F. Klostermeyer and C. M. Mynhardt}. Protecting a graph with mobile guards. \textit{Appl. Anal. Discrete Math.} 10 (2016), no. 1, 1–29. 
 	
	\bibitem{klostermeyer2020eternal} \textsc{W. F. Klostermeyer and C. M. Mynhardt}. Eternal and secure domination in graphs. \textit{Topics in domination in graphs, Dev. Math.} 64 (2020), 445–478, Springer, Cham.
		
	\bibitem{lovasz1972characterization} \textsc{L. Lov{\'a}sz}. A characterization of perfect graphs. \textit{J. Combin. Theory Ser. B} 13 (1972), 95–98. 
 	
	\bibitem{lovasz1972normal} \textsc{L. Lov{\'a}sz}. Normal hypergraphs and the perfect graph conjecture. \textit{Discrete Math.} 2 (1972), no. 3, 253–267.  	
	
	\bibitem{mckay2014practical} \textsc{B. D. McKay and A. Piperno}. Practical graph isomorphism, II. \textit{J. Symbolic Comput.} 60 (2014), 94–112.
	
	\bibitem{mycielski1955coloriage} \textsc{J. Mycielski}. Sur le coloriage des graphes. \textit{Colloq. Math.} 3 (1955), 161–162.  
 	
	\end{thebibliography}
\end{document}